\documentclass[pdflatex,sn-mathphys-ay]{sn-jnl}

\usepackage{graphicx}%
\usepackage{multirow}%
\usepackage{amsmath,amssymb,amsfonts}%
\usepackage{mathtools}
\usepackage{amsthm}%
\usepackage{mathrsfs}%
\usepackage[title]{appendix}%
\usepackage{xcolor}%
\usepackage{textcomp}%
\usepackage{manyfoot}%
\usepackage{booktabs}%
\usepackage{algorithm}%
\usepackage{algorithmicx}%
\usepackage{algpseudocode}%
\usepackage{listings}%
\usepackage{comment}

\newcommand{\der}{\mathrm{d}}
\newcommand{\extder}{\mathbf{d}}
\newcommand{\coder}{\extder^*}

\DeclareMathOperator{\image}{Im}
\DeclareMathOperator{\kernel}{Ker}
\newcommand{\nforms}[1]{\Omega^{#1}\manifold}
\newcommand{\harmonicpr}{\mathbf{H}}
\newcommand{\greenop}{\mathbf{G}}
\newcommand{\multiplyfunc}[1]{\mathfrak{m}({#1})}

\newcommand{\prob}{\mathbb{P}}
\newcommand{\Expectation}{\mathbb{E}}
\DeclareMathOperator{\Var}{Var}

\newcommand{\manifold}{\mathcal{M}}
\newcommand{\gradient}{\nabla}
\newcommand{\divergence}{\nabla\cdot}

\newcommand{\tanbdl}{T\manifold}
\newcommand{\cotbdl}{T^*\manifold}

\newcommand{\hodgelplc}{\Delta}

\newcommand{\leray}{\mathbf{\Pi}}
\newcommand{\identity}{\mathbf{I}}
\newcommand{\quavar}[1]{\left< #1 \right>}

\newcommand{\norm}[2]{\left\lVert #1 \right\rVert_{#2}}
\newcommand{\LtwoMnorm}[1]{\left\lVert#1 \right\rVert_{L^2 (\manifold)}}
\newcommand{\lpmani}[2]{\left\lVert#1 \right\rVert_{L^{#2} (\manifold)}}
\newcommand{\hnormm}[2]{\left\lVert #1 \right\rVert_{H^{#2} (\manifold)}}

\DeclareMathOperator{\tr}{Tr}
\newcommand{\ltwotan}[1]{\left\lVert #1 \right\rVert_{L^2 (\manifold,\tanbdl)}}
\newcommand{\cknorm}[2]{\left\lVert #1 \right\rVert_{C^{#2}}}

\newcommand{\hsnorm}[1]{\left\lVert#1 \right\rVert_{\mathrm{HS}}}
\newcommand{\anorm}[2]{\left\lVert #1 \right\rVert_{H_A^{#2}}}
\newcommand{\opnorm}[1]{\left\lVert #1 \right\rVert_{L^2\rightarrow L^2}}

\newcommand{\lpnormprob}[2]{\left\lVert #1 \right\rVert_{L^{#2} (\prob)}}

\newcommand{\swhitenoise}{\eta}
\newcommand{\stwhitenoise}{\xi}

\newcommand{\cylindricalFunctions}{\mathcal{F}\mathcal{C}_b^2}

\newcommand{\generator}[1]{\mathcal{L}_{#1}}

\newcommand{\ulimit}{u^{ (\infty)}}
\newcommand{\newprob}{\tilde{\mathbb{P}}}
\newcommand{\newexp}{\tilde{\Expectation}}

\newtheorem{defn}{Definition}[section]
\newtheorem{prop}[defn]{Proposition}
\newtheorem{lma}[defn]{Lemma}
\newtheorem{thm}[defn]{Theorem}
\newtheorem{cor}[defn]{Corollary}

\newtheorem{rmk}[defn]{Remark}
\newtheorem{definition}{Definition}%
\numberwithin{equation}{section}

\begin{document}

\title[Scaling Limits of Stochastic Transport Equations on Manifolds]{Scaling Limits of Stochastic Transport Equations on Manifolds}


\author*[1]{\fnm{Wei} \sur{Huang}}\email{wei.huang@fu-berlin.de}



\affil*[1]{\orgdiv{Institute f\"{u}r Mathematik}, \orgname{Freie Universit\"{a}t Berlin}, \orgaddress{\street{Arnimallee 7}, \city{Berlin}, \postcode{14129}, \state{Berlin}, \country{Germany}}}




\abstract{
In this work, we generalize some results on scaling limits of stochastic transport equations on the torus, developed recently by Flandoli, Galeati and Luo in \cite{galeati2020convergence,flandoli2020convergence,
flandoli2024quantitative}, to manifolds. We consider the stochastic transport equations driven by  colored space-time noise (smooth in space, white in time) on a compact Riemannian manifold without boundary. Then we study the scaling limits of stochastic transport equations, tuning the noise in such a way that the space covariance of the noise on the diagonal goes to the identity matrix but the covariance operator itself goes to zero. This includes the large scale analysis regime with diffusive scaling.
    We obtain different scaling limits depending on the initial data. With space white noise as initial data, the solutions to the stochastic transport equations converge in distribution to the solution to a stochastic heat equation with additive noise. With square integrable initial data, the solutions to the stochastic transport equations converge to the solution to the deterministic heat equation, and we provide quantitative estimates on the convergence rate. 
    }




\maketitle

\section{Introduction}
    \subsection{Background and Motivation}
	The transport equation
	\begin{equation}
	\partial_t u_t (x)=-v_t (x)\cdot \nabla u_t (x)
	\end{equation}
	describes the advection of scalar quantities (chemical density, temperature) in space under incompressible flow with divergence free velocity field $v_t (x)$ (with non divergence-free velocity field the equation we should consider the continuity equation $\partial_t u_t (x)=- \nabla\cdot   (u_t (x)v_t (x))$). This equation can be solved using the characteristic line method and the solution is  the initial value transported by the flow generated by the velocity field.
	
	In this paper, we are going to discuss the properties of the stochastic transport equations, with the deterministic velocity field $v$ replaced by the time derivative of some $Q$-Wiener process $-W_t (x)$ which is still divergence-free, with quadratic variation $$\quavar{W^i (x),W^j (y)}_t=tQ^{ij} (x,y).$$
	
	We consider the following stochastic transport equation, 
	\begin{equation}
	\der u_t (x) = \circ \der W_t (x) \cdot \nabla u_t (x),
	\end{equation}
	where $W$ is a vector-valued noise term that drives the flow. The transport noise that is white in time and correlated in space plays a crucial role in fluid dynamics. It appears in the celebrated Kraichnan model in turbulence theory (see \cite{kraichnan1968small}) and arises naturally in various fluid equations derived from stochastic variational principles (see \cite{holm2015variational}).
    
    The stochastic integral has to be understood in Stratonovich sense so that the solution can still be given as the transport of the initial value, since the characteristic line method relies on the chain rule. Another reason for using Stratonovich integrals instead of It\^o ones is that we want to consider the problem on manifolds, which will involve change of coordinates.
	
	The well-posedness of the stochastic transport equations with noises that are regular enough in space has been well studied a long time ago (see for example \cite{bismut1982mecanique,kunita1986lectures,kunita1990stochastic}), and even with lower regularity in time, well-posedness can be established with rough path theory as long as the space regularity is nice (see \cite{bellingeri2021transport}).  But the equation is ill-posed when the noise is too irregular in space. In fact, when $\Expectation \lVert W_1 \rVert_{L^2}$ is infinity, the It\^o-Stratonovich corrector will become ``$+\infty \Delta u$'', which would force the solution to mix immediately to a constant. A rigorous argument can be found in \cite{flandolikolmogorov} (Theorem 1.3).
	
	In order to make sense of the equation in that case, we have to mollify and rescale the noise. 
    In a series of recent works by Flandoli, Luo and Galeati, (\cite{flandoli2020convergence}, \cite{galeati2020convergence}, \cite{flandoli2024quantitative}) they found a proper scaling to get non-trivial limits of stochastic transport equations and some other SPDEs with transport noises including stochastic 2D Euler equation (\cite{flandoli2021scaling}) and stochastic 2D inviscid Boussinesq equations (\cite{luo2021convergence}). 

    We have much freedom in choosing the scaling of noises to maintain scaling limits, among which an important case is to study the large scale behavior, similar to those in the study of critical and supercritical SPDEs, for example 2D Anisotropic KPZ  and Burgers for $d\geq 2$ in \cite{gu2020gaussian}, \cite{cannizzaro2023weak}, \cite{cannizzaro2023gaussian}. More precisely, it falls into the regime for studying supercritical SPDEs like  3D Burgers in \cite{cannizzaro2023gaussian} with diffusive scaling. There will be a more detailed discussion later in this section and in Section \ref{scaling}.
    
	In \cite{galeati2020convergence} and \cite{flandoli2024quantitative}, they considered the stochastic transport equations on the torus (they also considered more complicated equations with nonlinear terms or deterministic drift, while the simplest case still reveals the key ingredient of the proofs).
	\begin{equation}
	\der u^{ (N)}_t= \circ \der W^{ (N)}_{t} \cdot \nabla u^{ (N)}_t,
	\end{equation}
	with the noise scaled like $\der W^{ (N)}=\Theta^{ (N)}\xi$, where $\Theta^{ (N)}$ are Fourier multipliers and $\xi$ is the vector-valued space-time white noise, or in other words $$W^{ (N)}= \sum_k \theta_k^{ (N)}\Psi_k (x)B^k,  $$ where $\Psi_k$ are some vector times the Fourier modes that form an orthonormal basis of the space of divergence-free vector fields, and $B^k$ are independent Brownian motions. $\theta_k^{ (N)}$ are real numbers and $\theta_k^{ (N)}=\theta_l^{ (N)}$ when the corresponding Laplace eigenvalue of $\Psi_k,\Psi_l$ are the same.
	
	They showed that if the $l^2$ norms of $\theta^{ (N)}$ are fixed  (or converge to some positive value), and the $l^{\infty}$ norms go to 0, then the solutions converge to a parabolic equation. If the initial value is regular enough (see \cite{galeati2020convergence}), then the limit is the solution to a deterministic heat equation while if the initial value is space white noise, the limit is a stochastic heat equation with an additive noise. In the regular case, there is  a quantitative estimate of the convergence rate in \cite{flandoli2024quantitative} with $L^2$ initial value. In the white noise case, the convergence is given by tightness argument as in \cite{flandoli2020convergence} (they showed the convergence of stochastic Euler equation, but the techniques can be easily adapted to stochastic transport equation), so estimates on the convergence rate are unavailable. 
The scaling limit result is interpreted as the eddy dissipation emerging from small scale turbulence in fluid dynamics (see e.g. \cite{majda1999simplified}) in \cite{flandoli2024quantitative,flandoli2023stochastic}.

    We would also like to mention the connection between SPDEs with transport noises and interaction particle systems in environmental noises, which provides some intuitions for the additive noise that appears in the scaling limit of white noise solutions. We consider a particle system with non-interacting particles driven by environmental noises, \begin{equation*}
	    \der X^{(i)}_t= \circ \der W_t(X_t^{(i)}).
	\end{equation*} The empirical measure $u_t=\frac{1}{N}\sum_{i=1}^N \delta_{X^{(i)}_t}$ satisfies the  stochastic transport equation.
    Assume that $\frac{\sum_{i=1}^N \delta_{X^{(i)}_0}}{N} \rightarrow u_0$ for some regular initial data $u_0$, which can be obtained by LLN with IID starting points or inhomogeneous Poisson point processes. The limit, $u_t=\lim_{N\rightarrow \infty}\frac{\sum_{i=1}^N \delta_{X^{(i)}_t}}{N}$, satisfies the stochastic transport equation with regular initial value. It can be viewed as the simplest case of mean field limit of particle systems with environmental noises in \cite{coghi2016propagation}, with no interaction at all. In 
 \cite{guo2023scaling}, they treat the scaling limit while scaling the noise (in the same way as in the scaling limits of SPDEs with transport noises in \cite{flandoli2024quantitative}) and number of particles at the same time.
    
    The white noise solutions can be viewed as fluctuations of the particle system starting from a uniform Poisson point process (see \cite{shao2025fluctuation}). 
    We can also obtain the white noise solution by fluctuation of  particles with i.i.d. weights  (see \cite{Flandoli2018weak} for stochastic 2D Euler), with $u_t=\lim_{N\rightarrow\infty}\sum_{i=1}^N\frac{w_i \delta_{X^{(i)}_t
    }}{\sqrt N}$, where $w_i$ are mean zero square integrable i.i.d. random variables that are independent of the particles and the noise.  The scaling limit of the white noise solution (in \cite{flandoli2020convergence} or this paper) is  the fluctuation of particle systems with idiosyncratic noises as in \cite{wang2023gaussian}.   It is not surprising since the noises are scaled such that they become independent Brownian motions in the limit. The additional Laplacian in the limit equation corresponds to the generator of the Brownian motion and the additive noise in the limit corresponds to the martingale term in \cite{wang2023gaussian}.

    Among all the possible choices of scaling, a family of choices are related to the large scale behaviors of the stochastic transport equations, similar to the setup of stochastic Burgers in \cite{cannizzaro2023gaussian}. Suppose we have a large torus of size $\lambda$, with regularized noise $\dot W=\Theta (-\Delta)\leray \xi$ on the tangent bundle of the large torus, where $\leray$ is the Leray projection, $\Theta\in C^\infty ([0,\infty))\cap L^2 ( (0,\infty),x^{d/2-1}\der x)$.  We choose the noise so that they are isotropic and have similar energy spectrum and local correlation structure ($Q^\lambda(x,y)$). We require $\Theta$ to be in $L^2((0,\infty),x^{d/2-1})$ to meet the requirement in \cite{flandoli2024quantitative}, so that $\Theta(-\Delta)$ is Hilbert-Schmidt. 
    We consider the stochastic transport equation with white noise initial data on a large torus $$\der u^\lambda=\nabla u^\lambda \cdot\circ\der W^{\lambda}, \text{on} \:\lambda\mathbb{T}^d, $$
    and with diffusive scaling ${u}^{ (\lambda)}_t (x)=\lambda^{d/2} u^\lambda_{\lambda^2 t} (\lambda x)$, we pull it back to the unit torus. Then $\tilde{u}^{ (\lambda)}$ satisfies the stochastic transport equation
    \[\der {u}^{ (\lambda)}=\nabla {u}^{ (\lambda)} \cdot\circ\der W^{ (\lambda)} \text{on} \: \mathbb{T}^d, \]
    where $\dot W^{ (\lambda)}=\lambda^{-d/2} \Theta (-\lambda^{-2}\Delta)\leray \xi$. Here we use $\xi$ to denote the tangent vector valued space-time white noise on the unit torus. The noise term $W^{ (\lambda)}$ converges to zero ($\Theta (-\lambda^2\Delta)$ converges to $\Theta (0)$) but $\tilde{u}^{ (\lambda)}$ converges to $u$ with \[ \partial_t u= C \Delta u+(-2C\Delta)^{-1/2}\xi,\] with some constant $C$ proportional to $\| \Theta \|^2_{L^2 ( (0,\infty), x^{d/2-1}\der x)}$, since the constant in the scaling limit (\cite{flandoli2024quantitative,flandoli2020convergence})  is proportional to 
    \begin{align*}
        \sum_{x\in\mathbb{Z}^d}\lambda^{-d}\Theta^2(\lambda^{-2}|x|^2)\rightarrow\int_{\mathbb{R}^d}\Theta^2(|x|^2)\der x=C\| \Theta \|^2_{L^2 ( (0,\infty), x^{d/2-1}\der x)}.
    \end{align*}

    The additive noise in the limit are of the same form as in stochastic Burgers equations (\cite{cannizzaro2023gaussian}). 

    A similar large scale analysis also applies to $L^2$ solutions with slowly varying initial data (rescaled to keep the $L^2$ norm) $u^\lambda_0(\lambda x)=\lambda^{-d/2} u_0(x)$. 

    The scaling limits result with $L^2$ initial data can also be helpful in understanding the detection of coherent sets  (blocks of fluid that move with minimal dispersion) in geophysics  (\cite{banisch2017understanding,froyland2021detecting}). We can model the motion of the atmosphere or ocean in terms of a deterministic flow driven by a velocity field $b_t (x)$  plus some stochastic perturbation $\der W$ that is white in time and short range correlated in space. Then the coherent sets are ``stable'' sets under the perturbation, i.e. the sets that are not affected much by the perturbation.  
    Let $U_t$ be the time evolution of a set that transformed by the deterministic flow and $\tilde{U}_t$ be a set transformed by the perturbed flow. Then the indicator functions $1_{U_t}$ and $1_{\tilde{U}_t}$ satisfy the deterministic transport equation \[ \der u=- b \cdot \nabla u \]
    and stochastic transport equations
    \[ \der u=- (b+\circ \der W )\cdot \nabla u \]
    respectively. The scaling limit results indicate that when the spatial correlation of the noise decays fast, the solution to the stochastic transport equation is very close to the solution to the advection-diffusion equation \[ \der u=-b\cdot\nabla u+C\Delta u.\] Therefore, we can estimate the difference between $U_t$ and $\tilde{U}_t$ by studying the effect of the Laplacian in the equation.  In this particular problem, we need to consider the stochastic transport equation on the surface of Earth which is not a torus. Therefore, it is worthwhile to consider the generalization. Moreover, the techniques developed for stochastic transport equations on manifolds may also be applied to other SPDEs with transport noises, many of which are not on the torus in  applications. Large scale analysis is relevant in this context, as the characteristic length scales of atmospheric flows are significantly larger than the correlation length of small scale perturbations.
	
	In the previous works on the torus (\cite{galeati2020convergence,flandoli2020convergence,flandoli2021scaling,flandoli2024quantitative}), their methods rely on the Fourier series which is no longer available on general manifolds.  Moreover,  their condition for convergence cannot be adapted directly to manifolds, and we have to deal with non-constant It\^o-Stratonovich corrector, similar to the situation in \cite{flandoli2022eddy}.
    In this paper, we are going to generalize these results to compact Riemannian manifolds, which requires tools to replace the Fourier decomposition. We will use  heat kernel estimates on compact Riemannian manifolds. On  manifolds, due to lack of symmetry, it is difficult to force the covariance on the diagonal to be the identity matrix as on the torus, but we can still find a sequence of noises whose covariance on the diagonal converges to the identity matrix. Compared to \cite{flandoli2022eddy} where they construct a sequence of noises that has a lower bound of covariance which diverges to infinity, we get a sequence of noises whose covariance on the diagonal converges to the identity matrix.
    We will construct the noises via renormalization of heat-kernel-mollified space-time white noise on general manifold. The covariance on the diagonal is not fixed but converges to the identity matrix. To show the convergence of the covariance on the diagonal, we need the asymptotic expansion of the heat kernel on Riemannian manifolds.
    
    The results can be easily generalized further to  smooth compact manifolds with a canonical volume form or density, which also include symplectic manifolds.  We only need the Riemannian metric for technical reasons, as the transport equation is well defined on general smooth manifolds, and our techniques work as long as the conservation of $L^2$ norm holds, which can be obtained when the vector valued noise is divergence-free with respect to the volume. 
    However, we can artificially construct a Riemannian metric such that the volume induced by the Riemannian structure is the same as the canonical volume form. The choice of such Riemannian structures is not unique, but the norms defined with different Riemannian metrics are equivalent on compact manifolds. Therefore, our results are the same with possibly different constants in the estimates depending on the Riemannian metric.

	\subsection{Main results and structure of the paper} 
	
    In this paper, we will consider the stochastic transport equation,
    \begin{equation*}
	\der u_t (x) = \circ \der W_t (x) \cdot \nabla u_t (x),
	\end{equation*} on a $d$-dimensional smooth compact Riemannian manifold $\manifold$,
	where $W$ is a vector-valued noise term  that drives the flow. We will generalize the scaling limits results on the torus in previous works (\cite{galeati2020convergence,flandoli2020convergence,flandoli2024quantitative}) to manifolds. The main results will be presented in Section 4. We will prove that if we tune the noises such that $A^{ (N)} (x)=Q ^{(N)}(x,x)$ converges  to $g^{-1}$ in $C^\infty (\manifold, \tanbdl\otimes\tanbdl)$ (or converge to the identity matrix in $C^\infty(\cotbdl\otimes\tanbdl)$ if we identify the tangent bundle and cotangent bundle using the Riemannian metric) and $Q^{ (N)}$ converges to $0$ in $L^2 (\manifold\times\manifold, \tanbdl\boxtimes\tanbdl)$, then the solutions to the stochastic transport equations converge to solutions to heat equations. 
    
    The conditions in this paper cover the situation on the torus, while we do not limit our scaling limits results to isotropic noises. We cannot directly use the conditions in the previous works due to lack of symmetry on general manifolds.  The symmetry on torus guarantees that $Q^{ (N)} (x,x)$ is the identity matrix times a constant which does not depend on $x$ and is proportional to the trace of $Q^{ (N)}$ (=$l^2$ norm of $\theta^{ (N)}$).
    On manifolds the It\^{o}-Stratonovich corrector is $\divergence( A^{ (N)}\gradient)$, where $A^{ (N)} (x)=Q^{ (N)} (x,x)$. 
    To make it converge, we need the convergence of $A^{ (N)}$, while convergence of $l^2$ norms of $\theta^{ (N)}$ only guarantees it in some specific cases with nice symmetries. 
    On Riemannian manifolds, we would instead require that $A^{ (N)} (x)=Q^{(N)} (x,x)$ converges  to  $g^{ij}$ in $C^k (\manifold, \tanbdl\boxtimes\tanbdl)$ for some large enough $k$, and $Q^{ (N)}$ converges to $0$ in $L^2 (\manifold\times\manifold, \tanbdl\boxtimes\tanbdl)$ or equivalently the Hilbert-Schmidt norm of the operator (which we also denote by $Q^{(N)}$ by an abuse of notation) $\hsnorm{Q^{(N)}}$ converges to $0$. The convergence of Hilbert-Schmidt norm can also be replaced by the operator norm $\|Q^{(N)}\|_{L^2\rightarrow L^2}$ converges to $0$, since \[\|Q^{(N)}\|_{L^2\rightarrow L^2}\leq \hsnorm{Q^{(N)}}=\|Q^{(N)}\|_{L^2(\manifold\times\manifold)}, \]
    $$\hsnorm{Q^{(N)}}\leq \opnorm{Q^{(N)}}^{1/2}|\tr Q^{(N)}|^{1/2},$$ and $$\tr Q^{(N)}=\int_\manifold \tr A^{(N)}$$ is uniformly bounded when $A^{(N)}$ is uniformly bounded in $C(\manifold)$.  The conditions are the same as in the previous works (\cite{flandoli2024quantitative},\cite{galeati2020convergence}) on the torus with isotropic noises, but our new conditions do not require noises to be of that form. For non-isotropic noises on a domain that is not a torus, there are similar conditions in \cite{flandoli2022eddy}, where they let $A^{(N)}$ converge to infinity. The condition in our paper would indicate that the covariance on the diagonal to be almost fixed, while the correlation in different space points should vanish. Particles in the stochastic flows driven by those noises will become independent Brownian motions in the limit.  On general Riemannian manifolds, it is not known whether we can take a sequence of divergence-free noises such that their covariance are some constant times the identity matrix as on the torus. We will revisit the situation on the torus in Section \ref{torusscal} and give an example of proper scaling on general manifolds via mollified space-time white noise in Section \ref{scaling}, which falls into the large scale analysis regime.  

    We treat two different types of initial data, white noise initial data and square integrable initial data in Section \ref{s-cvg-wn} and Section \ref{s-cvg-regular}.
    
    With tightness argument, we prove  that the white noise solutions to the stochastic transport equations as processes in $C^\alpha ([0,T],H^{-d/2-2\alpha-})$ $  (\alpha\in[0,1/2))$  converge in law to  the stationary solution to the stochastic heat equation
    \[ \partial_t u=\frac{1}{2}\Delta u+  (-\Delta)^{1/2}\xi,\]
    where $\xi$ is the space-time white noise on the manifold.  

    When the initial condition is in $L^2$, the limit instead becomes solution to the deterministic heat equation
    \[ \partial_t u=\frac{1}{2}\Delta u.\] On manifolds, or on the torus with non-isotropic noises, we have an additional deterministic error term, which is the result of the difference between $Q (x,x)$ and the identity matrix. We can get quantitative estimates on the convergence rate of the stochastic part in Theorem \ref{thm-regularity-moment}, which is of the same order as in \cite{flandoli2024quantitative} and the rate is the  best to expect due to the CLT result in \cite{galeati2023ldp}. We also get the regularity in time of order $\alpha \in[0,1/2)$ in exchange of $2\alpha$ regularity in space while the previous works only considered continuity in time or $L^\infty$ bound in time. We get the $p$-th moment bound of $C^\alpha ([0,T], H^{-d/2-2\alpha-\varepsilon})$ norm of the stochastic part  (Corollary \ref{quantitative-stoch}). 

    We prove the results for $A$ converging to the identity matrix, but similar results can be obtained with the same techniques if the limit is a positive-definite-matrix-valued function $A^{ (\infty)}$ on $\manifold$. In this case, the Laplacian in the limit should be replaced by the divergence form second order elliptic operator $\nabla\cdot (A^{ (\infty)}\nabla)$. The further generalization to compact manifold with canonical density or volume can be obtained from this result. 

    \paragraph{Structure of the paper}
    We state some results in Section 2 from differential geometry that are needed to establish our results, among which the asymptotic expansion of heat kernel will be needed to construct a proper scaling on general manifolds. We also recall the stochastic integral in infinite dimensions, in a language slightly different to \cite{da2014stochastic}. We also prove some technical lemmas. In Section 3 we discuss the solution theory of the stochastic transport equation.  We also prove (probabilistically) strong well-posedness but (probabilistically) weak existence and prior estimates are sufficient for getting the scaling limits. We also prove regularity estimates  and some important properties like stationarity of white noise and   $L^2$ conservation, which are used in getting the scaling limits. The main results are stated and proved in section 4.

 \section{Preliminaries}

 \subsection{Differential geometry basics}

In this subsection, we will list the definitions and notations in differential geometry to establish the basic setup of the stochastic transport equations on a Riemannian manifold. Proofs can be found in basic differential geometry and Riemannian geometry textbooks like \cite{lee2013smooth},\cite{warner1983foundations}, \cite{petersen2016riemannian} and  \cite{lee2018introduction}. 

In the paper we will consider the stochastic transport equation on a given smooth compact Riemannian manifold without boundary $ (\manifold,g)$ of dimension $d\geq 2$, where $g (\cdot, \cdot)$ is the Riemannian metric, a smooth positive-definite inner product on the tangent bundle. In a local coordinate system, $g=g_{ij} \extder x^i \otimes \extder x^j$. For two vector fields $X,Y$, the inner product of them $X\cdot Y=g (X,Y)$ is given by $g_{ij} X^iY^j$. We use Einstein summation convention, which implies summations over pairs of upper and lower indices, to simplify the notations. 

We denote the tangent bundle by $\tanbdl$, which is canonically isomorphic to the cotangent bundle $\cotbdl$. We will identify the tangent and cotangent bundles with the canonical isomorphism given by the Riemannian metric $g$ via Riesz representation. In local coordinates, a vector field $X=X_i \frac{\partial}{\partial x^i}$ can be also viewed as the covector or 1-form $w=w_i\extder x^i$, where $w_i=g_{ij}X^j$. We will also do the same trick to their tensor products, for example $\cotbdl \otimes \tanbdl \cong \tanbdl \otimes \tanbdl$. In other words, we omit the musical notations, namely the sharps and flats which raise or lower the indices, if they can be indicated from the context.

There is a canonical measure determined by the Riemannian metric on the manifold. Throughout this article,  all the integrations of functions are integrated with respect to that measure. In coordinates, it can be expressed as $\sqrt{|g|}\der x^1 \cdots \der x^n$, where $|g|$ is the determinant of $g_{ij}$.  

The gradient of a function $f$ on the manifold is defined as for any smooth vector fields $X$, $\nabla f \cdot X= \extder f  (X)$. In coordinates, $ (\nabla f)^i=g^{ij}\partial_j f$, where $g^{ij}$ is the inverse of $g_{ij}$. The divergence is defined as the formal dual of minus gradient, which in coordinates can be written as $\nabla \cdot X=|g|^{-1/2}\partial_i  (X^i|g|^{1/2})$.

We define the Laplace-Beltrami operator $\Delta$ on manifolds to be the divergence of gradient, namely $\Delta=\nabla \cdot \nabla$. 
\begin{prop}
	Let $X$ be smooth vector field and $u,v$ are smooth functions on $\manifold$. Then
	\begin{enumerate}
		\item $\nabla (uv)=u\nabla v+ v\nabla u, \nabla \cdot  (uX)=\nabla u \cdot X + u \nabla \cdot X$.  (Leibniz rule)
		\item $\Delta (uv)= (\Delta u)v+2\nabla u\cdot \nabla v + u (\Delta v)$.  (Leibniz rule for the Laplacian)
		\item $\int_{\manifold}\nabla \cdot X (x)\der x=0$.  (Stokes' theorem)
		\item $\int_{\manifold} X (x) \cdot \nabla u (x)\der x=-\int_{\manifold}  (\nabla \cdot X) (x)u (x)\der x$.  (Integration by parts)
		\item $\int_{\manifold} (\Delta u (x))v (x)\der x = -\int_{\manifold} \nabla u (x) \cdot \nabla v (x) \der x= \int_{\manifold}u (x) (\Delta v (x))\der x$.  (Symmetry of $\Delta$)
	\end{enumerate}
\end{prop}

\begin{rmk}
	The formulation of the stochastic transport equations does not depend on the metric once the noises are properly given. We need to give a specific metric to in order to have a reasonable choice of noises: the white noise on the tangent bundle and the mollified white noises. The divergence free condition on the driving noises gives a nice weak formulation and an It\^o-Stratonovich corrector in divergence form, which  only makes sense if there is a canonical density.
\end{rmk}

We denote the space of smooth $n$-forms by $\nforms{n}$. 

The inner product (denoted by `` $\cdot$ '') on a vector space $V$ induces a natural inner product on its exterior algebra $\Lambda^n (V)$, such that 
\[ (v_1\wedge\dots\wedge v_n)\cdot (w_1\wedge\dots \wedge w_n)=\det (v_i\cdot w_j)_{ij}. \]

Therefore, we can equip the space of $n$-forms with an inner product. For any $X,Y$ being sections (smooth functions with $x$ mapping to the fiber at $x$, for example, a section of the tangent bundle is a smooth vector field) of some inner product vector bundle, we define $ (X,Y)=\int_{\manifold}X (x)\cdot Y (x)\der x$, where the underlying measure is the one given by the Riemannian metric on $\manifold$. It induces a norm on the space of sections, which can be extended to $L^2$ functions and $ (\cdot,\cdot)$ will also denote the extension of inner product to the $L^2$ space.

We denote the dual of exterior derivative $\extder: \nforms{n}\rightarrow \nforms{n+1}$ by $\coder$. For $1$-forms, $\coder$ is  taking the minus gradient as vector fields. 

There is a natural connection induced by the Riemannian metric on the tangent space, called Levi-Civita connection. This connection can be extended naturally to the tensors and differential forms. With the Levi-Civita connection, we can take covariant derivatives on the manifold and do parallel transports of vectors along curves. A geodesic is a curve whose second order derivative is zero. With ODE theory we know that there is exactly one geodesic with given starting point and initial velocity. Suppose $p$ is a point in $\manifold$, and $X$ is a tangent vector at $p$. Then we define $\exp_p X$ to be the geodesic starting at $p$ with initial velocity $X$ at time 1. The exponential map is a local diffeomorphism, which induces local coordinates called normal coordinates on a neighborhood of $p$, which is the composition of the inverse of exponential map and an  isometry from $T_p\manifold$ to $\mathbb{R}^n$. 
Under the normal coordinates centered at $p$, $g$ at the origin of ($p$ on $\manifold$) is the identity matrix and its first order derivatives at $p$ are zero. Therefore, the Christoffel symbols are zero at $p$. The distance $d(x,y)$ of two points on the manifold is defined to be the infimum of the length of curves connecting them: $d(x,y)=\inf_{\gamma: \gamma_0=x,\gamma_1=y} \int_0^1|g(\dot\gamma_t,\dot\gamma_t)|^{1/2}\der t$. 

\subsection{Laplacians and Sobolev norms on Riemannian manifolds}

By \cite{taylor2011pseudodifferential}, any symmetric second order elliptic operator of the form $Lf=\nabla \cdot (A\nabla f)$ is a self-adjoint operator on $L^2(\manifold)$, where $A$ is a smooth symmetric 2-tensor that are positive definite at every point. Its eigenfunctions are smooth and we can choose a set of eigenfunctions that forms a ONS $\{\phi_j\}$ of $L^2(\manifold)$ and $-L\phi_j=\lambda_j\phi_j$. Then we can define the functional calculus of $-L$ by $f(-L)=\sum_jf(\lambda_j)(\cdot,\phi_j)\phi_j$, and for details of functional calculus, see \cite{einsiedler2017functional}.

In this paper, we will use the Sobolev norms defined by Bessel potential. Suppose $\alpha\in\mathbb{R}$, we define the Sobolev space $H^\alpha (\manifold)$ on the manifold to be $ (1-\Delta)^{-\alpha/2}{L^2 (\manifold)}$, equipped with the norm 
\begin{equation}
\hnormm{\phi}{\alpha}=\LtwoMnorm{ (1-\Delta)^{\alpha/2}\phi}=\left ( \phi,  (1-\Delta)^{\alpha}\phi \right)^{1/2}.
\end{equation}

By standard elliptic PDE theory (\cite{evans2010partial}), 
 we can also replace the $\Delta$ in the definition by some other elliptic operator $L$ defined by $Lf=\nabla\cdot(A\nabla f)$, where $A$ is a smooth uniformly elliptic 2-tensor field, to get another equivalent norm denoted by $\anorm{\cdot}{\alpha}$, with the controlling constants depending on the norms of $A$, $A^{-1}$ and their derivatives. To see this, we only need to verify the case when $\alpha$ is a positive even number, while the other cases can be obtained by interpolation or taking the dual. We fix a finite open cover $\{U_j\}$ of the manifold so that each open set there's a local coordinates $\varphi_j:U_j\rightarrow\mathbb{R}^d$. We choose a partition of unity $\{\phi_j\}$ such that the  $\operatorname{Supp} \phi_j \in U_j$. For any function $f$ on the manifold,  we can locally push forward $f\phi_j$ to a  function $(f\phi_j)\circ\varphi_j^{-1}$ on $\mathbb{R}^d$, whose support is in $\varphi_{j}(\operatorname{supp} \phi_j)$. Summing up the usual Sobolev norms of those functions, we get an alternative norm, which controls $\hnormm{\cdot}{\alpha}$. The equivalence of this norm and $\anorm{\cdot}{\alpha}$ follow by standard elliptic estimates (for example in \cite{evans2010partial}, using the interior regularity estimates iteratively) which also give a bound for the constants in terms of the norms of $A$,$A^{-1}$ and their (higher order) derivatives. Note that $\hnormm{\cdot}{\alpha}$ is  a special case of $\anorm{\cdot}{\alpha}$, so the norms are also equivalent. Besides,  the norm defined via local coordinates with different partitions and choices of coordinates are all equivalent. Those norms are also equivalent to the usual Sobolev norm defined with the covariant derivatives if $\alpha$ is an integer. The Sobolev embeddings and the product estimates on manifolds follows by standard localization and partition of unity arguments from results in the Euclidean case (see \cite{bahouri2011fourier}). 

For other vector bundles we can define Sobolev spaces similarly either with a elliptic operator (e.g. Laplacians on the tangent bundle) or via local coordinates. With the same type of arguments, those norms are equivalent. The vector-valued Sobolev spaces locally behave like vector-valued Sobolev spaces on Euclidean spaces. 

\subsection{Hodge theory and the Leray projection}
We define the Hodge Laplacian (sometimes also called Laplace-de Rham operator) to be $\hodgelplc=-\extder\coder-\coder\extder$. The Hodge Laplacian commutes with $\extder$ and $\coder$. It is a negative semidefinite operator (the sign is different from the usual geometer's). On $\Omega^0\manifold=C^\infty(\manifold)$, the Hodge Laplacian agrees with the Laplace-Beltrami operator. 



The Laplacians are defined on smooth sections but we can extend the definition to distributions $\mathcal{D}'(\manifold,\nforms{n})$ by $(\Delta \omega,\omega')=(\omega,\Delta \omega')$, for any $\omega'\in \mathcal{D}(\manifold,\nforms{n}),\omega\in \mathcal{D}'(\manifold,\nforms{n})$.

\begin{thm}[Hodge decomposition]
	Let $\hodgelplc$ be the Hodge Laplacian. Then the space of all $n$-forms $\nforms{n}$ has an orthogonal decomposition 
	\begin{equation}
	\nforms{n} = \image \extder_{n-1} \oplus \image \coder_{n+1} \oplus \kernel \hodgelplc_{n},
	\end{equation} where the lower indices indicate the domain  (for example $\der_n$ is the exterior derivative operator on $n$-forms). Moreover, $\kernel \Delta_n$ is finite dimensional and 
	\begin{equation}
	\kernel \hodgelplc_n=\kernel \extder_n \bigcap \kernel \coder_n,
	\end{equation} 
	\begin{equation}
	\kernel \extder_n=\image \extder_{n-1}\oplus \kernel \hodgelplc_n= (\image \coder_{n+1})^\perp,
	\end{equation}
	\begin{equation}
	\kernel \coder_n=  \image \coder_{n+1}\oplus \kernel \hodgelplc_n= (\image \extder_{n-1})^\perp.
	\end{equation}
\end{thm}
The proof can be found in \cite{demailly1997complex} or \cite{warner1983foundations}.

We define $\harmonicpr$ to be the orthogonal projection onto the harmonic forms $\kernel \Delta_{n}$, and we define the Green's operator $\greenop$ by setting $\greenop \alpha$ to be the unique solution to $-\Delta \omega=\alpha-\harmonicpr \alpha$ in $ (\kernel \Delta_n)^\perp$.  The Green's operator commutes with all operators that commute with $\Delta$, including $\extder$ and $\coder$. $\lVert\greenop \alpha\rVert_{H^2}\lesssim \lVert \alpha\rVert_{L^2}$, so it can be extended to $L^2$ forms. 

For any $n$-form $\alpha$, we have the following decomposition
\begin{equation}
\alpha=\extder\coder \greenop \alpha + \coder \extder \greenop \alpha + \harmonicpr \alpha,
\end{equation}
where each of the three terms on the right hand side falls in the one of three components in the Hodge decomposition. Hence, the three terms are orthogonal. In particular, $\coder \extder \greenop \alpha + \harmonicpr \alpha$ is in $\kernel \coder_{n}$, and $ \extder\coder \greenop \alpha$ is orthogonal to $\kernel \coder_{n}$. Thus, the orthogonal projection  onto $\kernel \coder$, which we call the Leray projection and denote by $\leray$, can be expressed as $\leray=\coder \extder \greenop+ \harmonicpr$. Moreover, $\leray$ commutes with $\Delta$ and $\hodgelplc \leray=\leray \hodgelplc= -\coder \extder$.

\subsection{Heat kernel expansions and heat kernel estimates}\label{subsec heat kernel expansion}
Following a series of computations and theorems in \cite{berline1992heat} which are summarized in Theorem 2.30 of the book,
we have an asymptotic expansion of the heat kernel $p (x,y)$ with Hodge Laplacian on $1$-forms (in the reference the Laplacian can also be other generalized Laplacians) with 
\begin{equation}
p_t (x,y)\sim q_t (x,y)\sum_{k=0}^{\infty} t^k \Phi_k (x,y),
\end{equation}
where $q_t (x,y)= (4\pi t)^{-d/2}\exp (-d (x,y)^2/4t)$, 
$\Phi_k\in C^\infty(\manifold\times\manifold, \tanbdl\boxtimes \cotbdl)$, i.e. they are smooth functions and $\Phi_k(x,y)$ takes value in $T_x\manifold\otimes T_y^*\manifold$, which are given iteratively by some 1st order PDEs on a neighborhood of the boundary. $\Phi_k$ is not unique outside the diagonal but on the diagonal its value and derivatives are uniquely determined by the Riemannian manifold $\manifold$. As we do not need to know the concrete value of $\Phi_k$ for $k\geq 1$, we will not present it here. Readers interested in the computations may check the explicit formulas in \cite{berline1992heat}. We only need that $\Phi_0 (x,x)=\identity$, and that $\Phi_k$ are smooth.

Let $N>d/2$, 
\begin{equation}
    k_t^N (x,y)=q_t (x,y)\sum_{k=0}^{N-1} t^k \Phi_k (x,y).
\end{equation}
 Then for $k, l\in\mathbb{N}$.
\begin{equation}\label{eq: error of heat kernel expansion}
\lVert\partial_t^k ( p_t - k_t^N) \rVert_{C^l (\manifold\times\manifold)}=O (t^{N-d/2-l/2-k}).
\end{equation}

By restricting to the diagonal we get 
\begin{equation}
p_t (x,x)\sim  (4\pi t)^{-d/2}\sum_{k=0}^{\infty}t^k\Phi_k (x,x). 
\end{equation}
Now we prove some technical lemmas based on the heat kernel expansions.
We will  need the following lemma to get an asymptotic expansion of $\extder\coder e^{t\hodgelplc}$:
\begin{lma}\label{lma-dd}
	Suppose $y\in \manifold$, $\omega$ is a 1-form, and in a normal coordinate system around $y$, $f_t (x)=e^{-|x|^2/4t}$. Then at the point $y$, $$\coder \extder  (f_t\omega)=  \frac{d-1}{2t}\omega+\tilde{\omega},$$ with $\tilde{\omega}=\coder\extder \omega$.
\end{lma}

\begin{proof}
	We do the computation in normal coordinates. At $y$, note that $f=1, \partial_i f=0,\partial_i\partial_j f=-\delta_{ij}/2t$, (we omit the lower index $t$ of $f$)
	\begin{align*}
	\coder\extder  (f\omega)_i=&-\sum_j\nabla^2_{jj} (f\omega_i)+\sum_j\nabla^2_{ji} (f\omega_j)
	\\=&-\sum_j\left ( (\partial_{jj}^2f)\omega_i+f\nabla^2_{jj}\omega_i\right) + \sum_j\left (  (\partial^2_{ji}f)\omega+f\nabla^2_{ji}\omega_j)\right)
	\\=&\omega_i\frac{d}{2t}-\sum_j\nabla^2_{jj} \omega_i -\frac{1}{2t}\omega +\sum_j \nabla^2_{ji}\omega_j
	\\=&\omega_i\frac{d-1}{2t}+\tilde{\omega}_i.
	\end{align*}
 \end{proof}

Then we get the asymptotic expansion of the kernel of $\leray e^{t\Delta}$, which is essential in the construction of the proper scaling in Theorem \ref{thm scaling via mollified space time white noise}. 
\begin{thm}\label{thm: diagonal leray heat kernel converge}
Let $\tilde{p}_t(x,y)$ be the integral kernel of the operator $\leray e^{t\Delta}$ on 1-forms. Let $a_t(x)=\tilde{p}_t(x,x)$. 
Then for any positive integer $l$, we can find a positive constant $C=C(\manifold,l)$, such that
\begin{equation}
    \left\|a_t -  (4\pi t)^{-d/2}\frac{d-1}{d}\identity \right\|_{C^l}\leq C( t^{-d/2+1}+1).
\end{equation}
\end{thm}

\begin{proof}
For $t\geq 1$, the heat kernel $p_t$ and its derivatives are uniformly bounded in space and time. Therefore we only need to show when $t\in(0,1)$,
\begin{equation}
    \left\|a_t -  (4\pi t)^{-d/2}\frac{d-1}{d}\identity \right\|_{C^l}\leq C t^{-d/2+1}
\end{equation}
Note that $\partial_t \leray e^{t\Delta}=\leray\Delta e^{t\Delta}=-\coder\extder e^{t\Delta}$.

As $a_1$ is smooth, it suffices to show 
\begin{equation}\label{leray heat kernel need to prove}
    \cknorm{\partial_t \left(a_t-(4\pi t)^{-d/2}\frac{d-1}{2t}\identity\right)}{l}\lesssim t^{-d/2},
\end{equation}
for all $t\in(0,1)$.

    By (\ref{eq: error of heat kernel expansion}), for $N>d/2$
    we get 
    \begin{equation}
        \left\|\partial_t \tilde{p}-\sum_{k=0}^N\coder \extder(q_t\Phi_k)\right\|_{C^l}=\|\coder \extder (p_t-k^N_t)\|_{C^{l}}=O(t^{N-d/2-l/2-1}). 
    \end{equation}
    
We choose $N$ so that $N>l/2+1$.
    By Lemma \ref{lma-dd},
    $\coder \extder (q_t\Phi_k)(x,x)=(4\pi t)^{-d/2}((2t)^{-1}\Phi_k+\coder\extder\Phi_k)(x, x)$. 
    Let $\phi_k(x)=\Phi_k(x,x)$ and $\psi_k(x)=\coder\extder\Phi_k(x,x)$, then
    \begin{equation}
        \left\|\partial_t a_t+(4\pi t)^{-d/2}\sum_k((2t)^{-1}\phi_k+\psi_k) \right\|_{C^l}\leq C t^{N-d/2-l/2-2},
    \end{equation}    
    for some constant $C=C(\manifold,l,N)$.
    Note that $\phi_0(x)=\Phi_0(x,x)=\identity$, $\phi_k,\psi_k$ are smooth functions, we get
    \begin{align*}
        &\cknorm{\partial_t(a_t-(4\pi t)^{-d/2}\frac{d-1}{d}\identity)}{l}
        \\&\leq  \sum_{k=1}^{N-1}t^{-d/2+k-1}\cknorm{\phi_k}{l}+\sum_{k=0}^{N-1}t^{-d/2-k}\cknorm{\psi_k}{l}+ Ct^{N-d/2-l/2-1}
        \\& \lesssim_{\manifold,l} t^{-d/2},
    \end{align*}
    which is the bound (\ref{leray heat kernel need to prove}) we need.
\end{proof}



\begin{prop}\label{prop: finite trace}
    The trace of the operator $(1-\Delta)^{\alpha}$ on $L^2(\manifold)$ is finite for $\alpha<-d/2$. 
\end{prop}
\begin{proof}
    By standard heat kernel estimates results from \cite{berline1992heat},  $p_t(x,x)\lesssim_{\manifold} t^{-d/2}+1$. 
    $\tr(e^{t\Delta})\lesssim t^{-d/2}+1$ as $\tr(e^{t\Delta})=\int_\manifold p_t(x,x)\der x$.
    Note that 
    \[
    (1-\Delta)^\alpha=\frac
    {1}{\Gamma(-\alpha)}\int_0^\infty t^{-\alpha-1}e^{-t}e^{t\Delta}\der t,
    \]
    we have \[
    \tr((1-\Delta)^\alpha)\lesssim\int_0^\infty t^{-\alpha-1-d/2}e^{-t} \der t+\int_0^\infty t^{-\alpha-1}e^{-t} \der t<\infty,
    \]
    when $\alpha<-d/2$.
\end{proof}

Now we prove some technical lemmas derived from heat kernel estimates for divergence form second order elliptic operators $L=\nabla\cdot A\nabla$, where $A$ is a smooth positive-definite 2-tensor field. By \cite{taylor2011pseudodifferential}, the eigenfunctions of $L$ are smooth and they form an orthonormal basis of $L^2(\manifold)$.
\newcommand{\kernelqAalpha}[2]{q_{#2}^{A,#1}}
 For any $\alpha\in\mathbb{R}$, we define $\kernelqAalpha{\alpha}{t}$ to be the integral kernel of $(1-L)^\alpha e^{tL}$.
\begin{lma}\label{lma kernel expansion via eigenfunction}
    Let $\phi_j$ be an eigenfunction of $L$ with the corresponding eigenvalue $-\lambda_j$ and $\{\phi_j\}$ be an orthonormal system of $L^2$. Then the kernel $\kernelqAalpha{\alpha}{t}$ is given by
    \begin{equation}\label{eq kernel via eigenfuntions}
        \kernelqAalpha{\alpha}{t}(x,y)=\sum_j (1+\lambda_j)^\alpha e^{-t\lambda_j}\phi_j(x)\phi_j(y),
    \end{equation}
    where the sum on the RHS converges in $H^\beta(\manifold)\otimes H^\beta(\manifold)$, for any $\beta\in\mathbb{{R}}$.
\end{lma}
\begin{proof}
    By \cite{berline1992heat} or \cite{taylor2011pseudodifferential}, the kernel of $e^{tL}$ is smooth. Therefore $\kernelqAalpha{\alpha}{t}$ is smooth and hence is in $L^2$. $\{\phi_i\boxtimes\phi_j\}$ form an ONS of $L^2(\manifold)\otimes L^2(\manifold)$ and $(\kernelqAalpha{\alpha}{t},\phi_i\boxtimes\phi_j)=(\phi_i,(1-L)^\alpha e^{tL}\phi_j)=(1+\lambda_j)^\alpha e^{-t\lambda_j}\delta_{ij}$, therefore $\kernelqAalpha{\alpha}{t}(x,y)=\sum_j (1+\lambda_j)^\alpha e^{-t\lambda_j}\phi_j(x)\phi_j(y)$ and the convergence is in $L^2(\manifold)\otimes L^2(\manifold)$. The same result also holds for $\kernelqAalpha{\alpha+\beta}{t}(x,y)=\sum_j (1+\lambda_j)^{\alpha+\beta} e^{-t\lambda_j}\phi_j(x)\phi_j(y)$ and the convergence is in $L^2(\manifold)\otimes L^2(\manifold)$. Note that $(1+\lambda_j)^{-\beta/2}\phi_j=(1-L)^{-\beta/2}\phi_j$ and $(1-L)^{-\beta/2}$ is in $L(L^2(\manifold),H^\beta(\manifold))$, the RHS of (\ref{eq kernel via eigenfuntions}) also converges in $H^\beta(\manifold)\otimes H^\beta(\manifold)$.
\end{proof}
\begin{cor}
\label{cor qaalpha expansion}
    Let $t>0$, $u\in H^{\beta}(\manifold)$. Suppose $Q\in L(L^2(\manifold,\tanbdl))$ is positive semidefinite and its integral kernel (by an abuse of notation also denoted by $Q$) is smooth. Then \[
    \tr_{H^\alpha_A} (e^{tL/2}\nabla^\ast\multiplyfunc{u}Q\multiplyfunc{u}\nabla e^{tL/2})=\int_{\manifold\times\manifold}(\nabla_x\nabla_y\kernelqAalpha{\alpha}{t})(x,y)\cdot Q(x,y) u(x) u(y)\der x \der y,
    \]
    where $\multiplyfunc{u}$ denotes the operator of multiplying $u$. When $\beta<0$, the integral on the RHS should be understood as a duality  pairing of $H^{-\beta}(\manifold)\otimes H^{-\beta}(\manifold), H^\beta(\manifold)\otimes H^\beta(\manifold)$.
\end{cor}
\begin{proof}
    Let $\phi_j,\lambda_j$ be the same as in  Lemma \ref{lma kernel expansion via eigenfunction}. Then
    \begin{align*}
        &\tr_{H^\alpha_A} (e^{tL/2}\nabla^\ast\multiplyfunc{u}Q\multiplyfunc{u}\nabla e^{tL/2})
     \\   =&\sum_j (1+\lambda_j)^\alpha(\phi_j, e^{tL/2}\nabla^\ast\multiplyfunc{u}Q\multiplyfunc{u}\nabla e^{tL/2}\phi_j)
     \\=& \sum_j e^{-t\lambda_j}(1+\lambda_j)^\alpha(\nabla\phi_j,\multiplyfunc{u}Q\multiplyfunc{u}\nabla\phi_j)
     \\=&\sum_j  e^{-t\lambda_j}(1+\lambda_j)^\alpha \int_{\manifold\times\manifold}\nabla\phi_j(x)\cdot Q(x,y)\nabla\phi_j(y) u(x) u(y) \der x \der y
     \\=&\int_{\manifold\times\manifold}(\nabla_x\nabla_y\kernelqAalpha{\alpha}{t})(x,y)\cdot Q(x,y) u(x) u(y)\der x \der y,
    \end{align*}
    where in the last step we use Lemma \ref{lma kernel expansion via eigenfunction}.
\end{proof}
\begin{lma} \label{lma:heat kernel estimate C2}
    Let $T>0$, and $\alpha\in (-d/2-1, -d/2)$. Let $q^{A,\alpha}_t$ be the kernel of $(1-L)^\alpha e^{tL}$, then for all $s,t$ satisfying $0<s\leq t\leq T$, we can find a positive constant $C=C(A,\manifold,\alpha, T)$, such that 
    \begin{equation}
        \| q^{A,\alpha}_t \|_{C^2(\manifold\times \manifold)} \leq C t^{-d/2-\alpha-1} .
    \end{equation} 
    \begin{equation}
        \| \kernelqAalpha{\alpha}{t} -\kernelqAalpha{\alpha}{s}\|_{C^2(\manifold\times \manifold)} \leq C (s^{-d/2-\alpha-1}-t^{-d/2-\alpha-1}) 
    \end{equation}
    Moreover, the constant $C$ can be chosen uniformly in $A$, provided that $A$ and $A^{-1}$ are uniformly bounded in the $C^k$ norm with some sufficiently large $k$ which may depend on $d$.
\end{lma}

\begin{proof}
    By standard heat kernel expansion techniques (see, e.g., \cite{taylor2011pseudodifferential}, \cite{berline1992heat}), one obtains that for any integers $l,j\ge0$
    \begin{equation}\label{eq:heat_kernel_bound}
        \| \partial_t^j q^{A,0}_t \|_{C^l(\manifold\times\manifold)} \lesssim_{A,\manifold,l} t^{-d/2-l/2-j} + 1,
    \end{equation}
    where the implicit constant can be chosen uniformly in $A$ provided that $A$ and $A^{-1}$ are uniformly bounded in the $C^k$ norm for some sufficiently large $k$ (which may depend on $l,d$).

    Next, by spectral calculus we have the representation
    \[
    (1-L)^\alpha = \frac{1}{\Gamma(-\alpha)} \int_0^\infty \tau^{-1-\alpha} e^{-\tau} e^{\tau L} \der \tau,
    \]
    so that
    \[
    (1-L)^\alpha e^{tL} = \frac{1}{\Gamma(-\alpha)} \int_0^\infty \tau^{-1-\alpha} e^{-\tau} e^{(t+\tau)L} \der\tau.
    \]
    
    Thus, the integral kernel $q^{A,\alpha}_t(x,y)$ of $(1-L)^\alpha e^{tL}$ is given by
    \[
    q^{A,\alpha}_t(x,y) = \frac{1}{\Gamma(-\alpha)} \int_0^\infty \tau^{-1-\alpha} e^{-\tau} \, q^{A,0}_{t+\tau}(x,y) \der \tau.
    \]

    To estimate the $C^2$ norm, we use the bound \eqref{eq:heat_kernel_bound} for $l=2$  and $j=0$:
    \[
    \|q^{A,\alpha}_t\|_{C^2} \lesssim_\alpha \int_0^\infty \tau^{-\alpha-1} e^{-\tau} (t+\tau)^{-d/2-1}\der \tau+  \int_0^\infty \tau^{-\alpha-1} e^{-\tau} \der \tau.
    \]
    The second term is finite and does not depend on $t$. Then we bound the first term by
    \begin{align*}
        &\int_0^\infty \tau^{-1-\alpha} e^{-\tau} (t+\tau)^{-d/2-1} \der \tau 
        \\\leq & t^{-d/2-\alpha-1} \int_0^\infty u^{-1-\alpha}(1+u)^{-d/2-1} e^{-ut} \der u
        \\ \leq& t^{-d/2-\alpha-1}\int_0^\infty u^{-1-\alpha}(1+u)^{-d/2-1}\der u.
    \end{align*}
    $\int_0^\infty u^{-1-\alpha}(1+u)^{-d/2-1}\der u $ is a finite constant that depends on $\alpha$ since $-1-\alpha>d/2-1\geq 0$ and
$-1-\alpha-d/2-1<-1$.   Hence, there exists a constant $C=C(A,\alpha,T)$ such that
    \[
    \|q^{A,\alpha}_t\|_{C^2(\manifold\times\manifold)} \le C\, t^{-d/2-\alpha-1}.
    \]
    
    For the time difference estimate, note that
    \[
    q^{A,\alpha}_t - q^{A,\alpha}_s = \frac{1}{\Gamma(-\alpha)} \int_0^\infty r^{-1-\alpha} e^{-r} ( q^{A,0}_{t+r} - q^{A,0}_{s+r} ) \der r.
    \]
    Applying the heat kernel estimates to the difference,  
    \begin{align*}
        &\| \kernelqAalpha{0}{t+r}-\kernelqAalpha{0}{s+r}\|_{C^2(\manifold\times\manifold)}
        \\ &\leq \int_s^t\|\partial_t \kernelqAalpha{0}{\tau+r}\|_{C^2(\manifold\times\manifold)}\der \tau
        \\ &\lesssim_{\manifold,A}  (s+r)^{-d/2-1}-(t+r)^{-d/2-1}+(t-s),
    \end{align*}
    one deduces that
    \begin{align*}
        &\|\kernelqAalpha{\alpha}{t}-\kernelqAalpha{\alpha}{s}\|_{C^2(\manifold\times\manifold)}
        \\ &\lesssim_{\manifold,A}  \int_0^\infty r^{-1-\alpha}e^{-r}((s+r)^{-d/2-1}-(t+r)^{-d/2-1}+(t-s))
        \\ &\lesssim\int_0^\infty
 r^{-1-\alpha}((s+r)^{-d/2-1}-(t+r)^{-d/2-1})\der r +(t-s)\int_0^\infty r^{-1-\alpha}e^{-r} \der r
 \\ &\lesssim_{\alpha}  s^{-d/2-1-\alpha}-t^{-d/2-1-\alpha} +(t-s)
 \\ &\lesssim_T s^{-d/2-1-\alpha}-t^{-d/2-1-\alpha}.
 \end{align*}
\end{proof}

\begin{lma}
    \label{lma: heat kernel difference trace bound}
     Let $T>0$, and $\alpha\in [-d/2-1, -d/2)$.Then for all $t\in(0,T]$, we can find a positive constant $C=C(A,\alpha, T,\manifold)$, such that 
    \begin{equation}
        \tr_{L^2}((1-L)^\alpha (1-e^{tL/2})^2)\leq C t^{-d/2-\alpha} .
    \end{equation} 
    Moreover, the constant $C$ can be chosen uniformly in $A$, provided that $A$ and $A^{-1}$ are uniformly bounded in the $C^k$ norm with some sufficiently large $k$ which may depend on $d$.
\end{lma}
\begin{proof}
    Note that $(1-e^{t\lambda/2})^2\leq 2(1-e^{t\lambda})$ for all $\lambda\leq 0$, and that $L$ is negative semidefinite,
    \begin{align*}
        &\tr((1-L)^\alpha (1-e^{tL/2})^2)
        \\ \leq&  2 \tr((1-L)^\alpha) (1-e^{tL}))
        \\=&2\int_0^t\tr((-L)(1-L)^{\alpha}e^{sL})\der s,
    \end{align*}
     where the trace is taken with respect to $L^2$.
    Note that trace of the operator $(-L)(1-L)^\alpha e^{sL}$ is the integral of its kernel on the diagonal. Therefore, the trace can be bounded by $\|A\|_{C^1}|\manifold|\|\kernelqAalpha{\alpha}{s}\|_{C^2}$. Then we apply Lemma \ref{lma:heat kernel estimate C2} to get
    
    \begin{align*}
        &\tr_{L^2}((1-L)^\alpha (1-e^{tL/2})^2)
        \\ &\lesssim_{A,\manifold}  \int_0^t \|\kernelqAalpha{\alpha}{s}\|_{C^2}\der s
        \\&\lesssim_{A,\manifold,T,\alpha} t^{-d/2-\alpha}.
    \end{align*}
\end{proof}


\subsection{Stochastic integrals on Hilbert spaces}
We will briefly recall the stochastic analysis on Hilbert spaces introduced in \cite{da2014stochastic}, with slightly different formulation.

Assume that $U$ is a Hilbert space and in this subsection we denote the inner product on it by $(\cdot,\cdot)_U$, or simply $(\cdot,\cdot)$ when there is no ambiguity. Assume $V,H$ are also Hilbert spaces. 
\begin{defn}
    \label{def gaussian}
    Let $W$ be a $U$-valued random variable, $Q\in L(U^\ast, U)$ and $Q^\ast=Q$. $W$ is called a centered Gaussian variable with covariance $Q$ if for any $f$ in $U^\ast$, $(f,u)$ is a centered Gaussian variable with variance $(f,Qf)$. 
\end{defn}
\begin{prop}
    Let $W$ be a $U$-valued centered Gaussian variable with covariance $Q$, and let $\Phi\in L(U,V)$. Then $\Phi W$ is a $V$-valued centered Gaussian variable with covariance $\Phi Q \Phi^\ast$.
\end{prop}
\begin{rmk}
\label{rmk quadratic natural}
    The covariance on $U$ naturally lies in $U\otimes U$, therefore it can be viewed as an operator from $U^\ast$ to $U$. It is very convenient to identify the Hilbert space $U$ with its dual space $U^\ast$ with the canonical isometry $\mathcal{J}: U\rightarrow U^\ast$, $e^\ast=(e,\cdot)_U$. , as in the book \cite{da2014stochastic}.  The covariance operator defined in their book is $Q\circ\mathcal{J}$. The problem with the identification is that we will work with Sobolev spaces and will identify $(H^\alpha)^\ast$ to $H^{-\alpha}$ instead of $L^2$. When $U=H^\alpha$, $U^\ast=H^{-\alpha}, \mathcal{J}=(1-\Delta)^{\alpha}$.  If we stick to our notation, when there is an embedding $\iota:H^\alpha \hookrightarrow H^\beta$, $\iota^\ast$ in our notation will be the natural embedding from $H^{-\beta}$ to $H^{-\alpha}$, while in \cite{da2014stochastic}, $\iota^\ast$ is $(1-\Delta)^{\beta-\alpha}$ from $H^{\alpha}$ to $H^\beta$. If we use our notation the covariance operator is the ``same'' in different spaces and their integral kernels are the same. But when we take the trace, we need to be careful and specify in which space we want to take the trace, and we take the trace of $Q\circ \mathcal{J}$. If not specified, we will take the trace of operators as they are on $L^2$ in Section 3 and Section 4. 
\end{rmk}
 \begin{prop} \label{prop Q Wiener existence}
     There exists $U$-valued centered Gaussian variable $X$ with covariance $Q$ if and only if $Q$ is positive semidefinite and $\tr_U Q<\infty$. Moreover, $\Expectation |X|^p\lesssim_p (\tr_U Q)^{p/2}$ for all $p\in[1,\infty)$. 
 \end{prop}
\begin{proof}
    See \cite{da2014stochastic}, Proposition 2.18 and Proposition 2.19.
\end{proof}

Now assume $Q$ to be trace class on $U$ ($Q\circ \mathcal{J}$ is trace class).
\begin{defn}
    \label{def Q Wiener}
    A $U$-valued stochastic process $W_t, t\geq 0$, is called a $Q$-Wiener process if
    \begin{itemize}
        \item $W_0=0$,
        \item $W$ has continuous trajectories,
        \item W has independent increments,
        \item $W_t-W_s$ is a centered Gaussian variable with covariance $(t-s)Q$, for any $t>s\geq 0$.
    \end{itemize}
\end{defn}

When $W_t$ is an $\mathcal{F}_t$-adapted $Q$-Wiener process, and $\Phi_t$ is a predictable process in the space of $L_0^2$, the  subspace of unbounded operator from $U$ to $V$ such that $\|\Phi_t\|_{L_0^2}=(\tr_{V} \Phi_tQ\Phi_t^\ast)^{1/2}$ is finite, we call $\Phi_t$ to be ``stochastically integrable'' on $[0,T]$, if $\int_0^T \|\Phi_t\|_{L_0^2}^2\der t $ is finite almost surely. By \cite{da2014stochastic}, the stochastic integral $\int_0^t \Phi_s \der W_s $ is well defined for $t\in[0,T]$ if $\Phi_t$ is stochastically integrable on $[0,T]$. The quadratic variation of the stochastic integral is given by $\int_0^t\Phi_s Q \Phi_s^\ast\der s$.

We will need the following estimate for the stochastic integrals on general Hilbert spaces:  

\begin{prop}
    \label{prop moment estimate via quadratic}
    For every $p\geq 2$, we can find positive constant $C$ such that for all  $t>s>0$
    \begin{equation}
        \lpnormprob{\norm{\int_s^t \Phi_r\der W_r}{V}}{p} \leq C \left(\int_s^t \lpnormprob{\tr_V (\Phi_r Q\Phi_r^\ast)}{p/2}\right)^{1/2}
    \end{equation}     
\end{prop}
The proposition above is a direct consequence of Theorem 4.37 in \cite{da2014stochastic}.

We will also need the It\^o's formula below, which is a special case of \cite{da2014stochastic}, Theorem 4.32.
\begin{prop}[It\^o's formula]
\label{prop ito's formula}
    Assume $\Phi$ is an $L_0^2$-valued process stochastically integrable in $[0,T]$, $\phi$ is a $V$-valued predictable process Bochner integrable on $[0,T]$, $\prob$-a.s., and $X_0$ is $\mathcal{F}_0$-measurable $V$-valued random variable.
    Then $X_t=X_0+\int_0^t\phi_s\der s+\int_0^t \Phi_s\der W_s, \: t\in[0,T] $ is well-defined. Suppose $F$ is second order Fr\'echet differentiable on $V$ and the first and second derivatives are uniformly continuous on bounded sets in $V$. Then
    
\[
F(X_t)=F(X_0)+\int_0^t DF(X_s)\Phi_s \der W_s+\int_0^t(DF(X_s)\phi_s+ \tr(D^2F(X_s)\Phi_sQ\Phi_s^\ast)) \der s,
\]
where $DF(x)$ is the first order derivative of $F$ at $x\in V$ which is in $V^\ast=L(V,\mathbb{R})$, $D^2F(x)$ is the second order derivative of $F$ at $x\in V$ which is in $ L(V,V^\ast)$ .  Note that $D^2F(X_s)\Phi_sQ\Phi_s^\ast\in L(V^\ast, V^\ast)$, and the trace is the usual trace of an operator on $V^\ast$ in contrast to the $\tr_U$ where we need to compose the operator with $\mathcal{J}$ as in Remark   \ref{rmk quadratic natural}.
\end{prop}
 
We will need the following proposition to compute the trace. For a function $u$ on the manifold, we denote the operator of multiplication (on some space of functions on $\manifold$) with $u$ by $\multiplyfunc{u}:=f\mapsto uf$, and for a tangent vector-valued function $X$, we denote the operator of multiplication (on some space of tangent vector-valued functions so that the product is well defined) by the inner product $\multiplyfunc{X\cdot}:=Y\mapsto X\cdot Y$.
\begin{prop}
    \label{prop trace uqu}
    Suppose $u\in L^2(\manifold)$,and  $Q\in C(\manifold\times\manifold , \tanbdl\boxtimes\tanbdl)$  is the kernel of a positive semidefinite operator $Q$ on $L^2(\manifold,\tanbdl)$. Then $u(x)Q(x,y)u(y)$ is the kernel of a semidefinite operator $\multiplyfunc{u}Q\multiplyfunc{u}$ and its trace is given by
    \[\tr(\multiplyfunc{u}Q\multiplyfunc{u})=\int_\manifold u(x)^2 \tr Q(x,x)\der x. \]
\end{prop}
\begin{proof}
    By Mercer's theorem, $Q$ has the following decomposition $Q(x,y)=\sum_k\Psi_k(x)\otimes\Psi_k(y)$, where $\{\Psi_k\}_k$ is a sequence of continuous vector fields and the sum converge absolutely and uniformly on $\manifold\times\manifold$. Let $A_N$ be the operator with kernel $u(x)u(y)\sum_{k=1}^N\Psi_k(x)\otimes\Psi_k(y)$, then 
    \begin{equation}
        \tr A_N=\sum_{k=1}^N \int_\manifold u(x)^2|\Psi_k(x)|^2\der x=\int_\manifold u(x)^2\tr(\sum_{k=1}^N\Psi_k(x)\otimes\Psi_k(x))\der x.
\end{equation}
The RHS converges to $\int_\manifold u(x)^2 \tr Q(x,x)\der x$ as the sum $Q(x,y)=\sum_k\Psi_k(x)\otimes\Psi_k(y)$ converges uniformly, and $u\in L^2(\manifold)$. $A_N$ converges to $\multiplyfunc{u}Q\multiplyfunc{u}$ in Hilbert-Schmidt norm as the kernel converges in $L^2$, and  that $A_N$ is increasing, $\tr A_N$ converges to $\tr \multiplyfunc{u}Q\multiplyfunc{u}$. 
\end{proof}

We will also need some properties of the trace class operators that can be found in classical functional analysis textbooks, for example \cite{reed1980methods}. 
\begin{prop}
\label{prop trace(AB) < opnorm A trace B}
    Let A,B be positive semidefinite operators on $U$. If $A$ is a bounded operator and $B$ is a trace class operator,  then $\tr(BA)=\tr(AB)\leq |A| \tr B$.
    Let $C$ be an bounded operator from $U$ to $V$. Then $\tr(CBC^\ast)\leq|C|^2\tr B$.
\end{prop}

\subsection{White noise on manifold}
Recall that a centered Gaussian process $\swhitenoise$ on the manifold is called a  (spatial) white noise if for any smooth function $\phi$,
\begin{equation}
\Expectation[ (\swhitenoise,\phi)^2]=\LtwoMnorm{\phi}^2.
\end{equation}
By polarization,  the covariance is given by
\begin{equation}
\Expectation[ (\swhitenoise,\phi) (\swhitenoise,\psi)]= (\phi,\psi)_{L^2 (\manifold)}, \:\forall\phi,\psi\in C^{\infty} (\manifold).
\end{equation}

The following integrability result follows directly from the Gaussian hypercontractivity.

\begin{prop} \label{prop-wn-moment}
	Let $\swhitenoise$ be a white noise on $\manifold$.
	For any smooth function $\phi$, $p\in [1,\infty)$, 
	\begin{equation}
	\Expectation | (\swhitenoise,\phi)|^p= C (p) \LtwoMnorm{\phi}^{p},
	\end{equation}
	for some constant $C (p)<\infty$.
\end{prop}

There are several equivalence characterizations of the white noise. An important one is given by the characteristic function.
\begin{prop}\label{prop-char-white}
	A random process $\swhitenoise$ is  space white noise if and only if for all smooth function $\phi$ on $\manifold$,
	\begin{equation}
	\Expectation[\exp ({i (\swhitenoise,\phi)})]=\exp\left (-\frac{\LtwoMnorm{\phi}^2}{2}\right),
	\end{equation}
\end{prop}

Another way to characterize the white noise is to look at its decomposition over an orthonormal basis of $L^2$, in particular the spectral decomposition of the Laplacian, with components being i.i.d. standard Gaussian. This is also why the noise is called ``white''. 

\begin{prop}
\label{prop regularit of wn}
	The space white noise $\swhitenoise$ on $\manifold$ lies almost surely in  $H^{-\frac{d}{2}-} (\manifold)$ $:=$ $ \cap_{\varepsilon>0}H^{-\frac{d}{2}-\varepsilon} (\manifold)$,
	and the $p$-th moment of its $H^{-d/2-\varepsilon}$ norm is finite. For any $\varepsilon>0, p \in [1,\infty)$,
	\begin{equation}
	\Expectation\left[\hnormm{\swhitenoise}{-\frac{d}{2}-\varepsilon}^p\right] < \infty.
	\end{equation}
\end{prop} 

\begin{proof}
	It is a consequence of Proposition \ref{prop: finite trace} and Proposition \ref{prop Q Wiener existence}.
\end{proof}

The law of white noise on $\manifold$ defines a measure on $H^{-\frac{d}{2}-} (\manifold)$, which we denote by $\mu$.

We will also need the following proposition that follows from standard Wiener chaos results.

\begin{prop}\label{prop-est-wn-quadratic}
	Suppose $\eta$ is a white noise, and $K (x,y)$ is a symmetric ($K (x,y)=K (y,x)$) smooth function on $\manifold\times\manifold$. Then we define $ (\eta\boxtimes\eta, K)$, also formally written as $\int_{\manifold\times\manifold} K (x,y)\eta (x)\eta (y)\der x \der y$, to be
	$ (\eta,K\eta)$,
	where by abuse of notation $K$ denotes the operator defined by $ (Kf) (x)=\int_{\manifold} K (x,y)f (y)\der y$.
    Then $(\eta\boxtimes\eta,K)=I_2(K)+\tr K=I_2(K)+\int_{\manifold}K(x,x)\der x$, where $I_2(K)$ is the multiple stochastic integral defined in \cite{nualart2006analysis}, Section 1.1.2.
	
	Moreover, the expectation of $ (\eta\boxtimes\eta, K)$ is given by the integral of $K$ on the diagonal, which is also the trace (in $L^2$) of the operator $K$
	\begin{equation}
	\Expectation\left[\int_{\manifold\times\manifold} K (x,y)\eta (x)\eta (y)\der x \der y\right]=\int_{\manifold} K (x,x) \der x = \tr K.
	\end{equation}
	
	The variace of $ (\eta\boxtimes\eta, K)$ is given by two times the square of the Hilbert-Schmidt norm of $K$, namely
	\begin{equation} \label{eq 2nd chaos var}
	\Var\left[  (\eta\boxtimes\eta, K) \right]=2 \int_{\manifold\times\manifold} K (x,y)^2 \der x \der y = 2\hsnorm{K}^2
	\end{equation}
	As $ (\eta\boxtimes\eta, K)-\tr K$ lies in second order Wiener chaos, for $p>2$,  the $p$-th moment can also be controlled:
	\begin{equation}\label{moment second chaos}
	\Expectation\left| (\eta\boxtimes\eta, K)-\tr K \right|^p \leq C_p \hsnorm{K}^p.
	\end{equation}
	
	Moreover, if $K$ is positive semidefinite, then for any $p\in[2,\infty)$ we can bound the $p$-th moment by the expectation, i.e. we can find constant $C_p$, such that 
	\begin{equation}\label{moment second integral}
	\Expectation\left[ (\eta\boxtimes\eta, K)^p\right] \leq C_p  (\tr K)^p.
	\end{equation}
	
\end{prop}
 \begin{proof}
     We will prove $(\eta\boxtimes\eta,K)=I_2(K)+\tr K$, and (\ref{eq 2nd chaos var}), (\ref{moment second chaos}) follow from \cite{nualart2006analysis} (Section 1.1.2 and 1.1.4). To get  (\ref{moment second integral}), we also use the fact that $\hsnorm{K}\leq \tr K$ for positive semidefinite $K$.
     
     When $K(x,y)=\phi(x)\phi(y)$ for some  $\phi\in L^2(\manifold)$, by Proposition 1.1.2 in \cite{nualart2006analysis}, $(\eta\boxtimes\eta,K)=(\eta,\phi)^2=I_2(\phi\boxtimes\phi)+(\phi,\phi)=I_2(K)+\tr K$. 
     Now we just need to show $K$ has a decomposition $K=\sum_k\phi_k\boxtimes\phi_k$, where the sum converges in $H^\alpha (\manifold)\otimes H^\alpha(\manifold)$ for some $\alpha>d/2$ as $\eta\in H^{-d/2-}(\manifold)$.  $(1-\Delta)^\alpha K(1-\Delta)^\alpha$ has a smooth kernel and symmetric,  and therefore it is Hilbert-Schmidt. We have $(1-\Delta)^\alpha_x (1-\Delta)^\alpha_y K(x,y)=\sum_{k}\varphi_k \boxtimes\varphi_k$, where the sum converges in $L^2(\manifold)\otimes L^2(\manifold)$. Let $\phi_k=(1-\Delta)^{-\alpha}\phi_k$. Then $K=\sum_k\phi_k\boxtimes\phi_k$ and the sum converges in $H^\alpha(\manifold)\otimes H^\alpha(\manifold)$.
 \end{proof}

\section{Stochastic transport equations}

In this section, we will discuss the well-posedness and regularity of  stochastic transport equations. We will also prove some priori estimates for $L^2$ solutions and white noise solutions. As well-posedness is not our main focus, we will assume the driving noise is smooth in space in order to simplify our proofs. The well-posedness of stochastic transport equations with regular enough initial data  has been well established in \cite{bismut1982mecanique} and also in \cite{kunita1986lectures,kunita1990stochastic} but we will also need to deal with white noise solutions that are in $H^{-d/2-}$.
\subsection{Formulations of  stochastic transport equations}
In this subsection, we discuss the formulations of the stochastic transport equations.  

We fix a complete filtered probability space $ (\Omega,\mathbb{P},\mathcal{F}, (\mathcal{F}_t))$. We assume the filtration to be right continuous for convenience. All stochastic processes are assumed to be adapted with respect to this filtration unless otherwise stated. The solutions discussed in this section are all strong solutions in the probability sense as they are adapted to the natural filtration generated by initial data and the noise, but for the scaling limits results, we only probabilistic weak existence. We will also need some regularity estimates of the stochastic transport equations.

\subsubsection{Assumptions on the noises}\label{noise-def}
Now we introduce some basic assumptions on the noises.
Let $Q$ be in $C^\infty (\manifold\times\manifold,\tanbdl\boxtimes\tanbdl)$, i.e. $Q$ is a smooth and $Q (x,y)$ takes value in $T_x\manifold\otimes T_y\manifold$ (the tensor product of tangent vector spaces at $x$ and at $y$). With an abuse of notation, we also use $Q$ to denote the operator on tangent vectors whose integral kernel is $Q$, i.e. $ (QX)^i (x)= \int_{\manifold}Q^i_j (x,y) X^j (y)\der y$. We assume $Q$ is positive semidefinite. We assume $W$  to be a $(\mathcal{F}_t)$-adapted $Q$-Wiener process. Note that for all $\alpha\in\mathbb{R}$, $\tr_{H^\alpha}(Q)=\tr((1-\Delta)^\alpha Q)=\int_{\manifold}\tr((1-\Delta)_x^\alpha Q(x,y)|_{y=x})\der x<\infty$. Therefore, $W$ is a $Q$-Wiener process in $H^\alpha$ for all $\alpha$. By Sobolev embedding, $W$ is smooth in space.

We assume $W$ is a divergence-free process, which is equivalent to $(\nabla\cdot)\circ Q\circ\nabla=0$, or $\partial_{x_i}\partial_{y_j} Q^{ij}(x,y)=0, \forall x,y \in \manifold$. It is also equivalent to $(\nabla \cdot) Q=0$, or $\partial_{x_i}Q^{ij}(x,y)=0, \forall x,y\in\manifold$. ($(\nabla\cdot)\circ Q\circ\nabla=-((\nabla \cdot) \circ Q^{1/2})((\nabla \cdot) \circ Q^{1/2})^\ast$, so $(\nabla\cdot)\circ Q\circ\nabla=0 \Leftrightarrow(\nabla \cdot)\circ Q^{1/2})=0\Rightarrow (\nabla\cdot)\circ Q=0$  )

When $Q=(\leray \Theta(-\Delta))(\leray\Theta(-\Delta))^\ast=\leray\Theta^2(-\Delta)$ and $\Theta(\lambda)\lesssim e^{-C\lambda}$ for some $C>0$, all the above conditions are satisfied.

The covariation of $ (W,X), (W,Y)$, where $X,Y$ are smooth deterministic vector fields, is given by $\left<  (W,X),  (W,Y) \right>_t= (X,QY)t$.

We denote the correlation on the diagonal as $A (x) = Q (x,x)$.

    \begin{align*}
    \partial_i A^{ij}(x)=&\partial_{x_i }Q^{ij}(x,y)+\partial_{y_i} Q^{ij}(x,y)|_{y=x}
    \\=&\partial_{x_i }Q^{ij}(x,y)+\partial_{x_i} Q^{ji}(x,y)|_{y=x}
    \\=&2\partial_{x_i }Q^{ij}(x,y)|_{y=x}
    \\=&0.
\end{align*}

In addition, we assume that the $A\geq C \identity$ for some $C>0$.

Then the operator $Lu=\divergence( A \gradient u)$ is a uniform elliptic second order differential operator. In coordinates $Lu=\partial_i( A^{ij} \partial_j u)=A^{ij}\partial_i\partial_j u$.

The additional uniform ellipticity above is not automatically satisfied on general manifolds when $Q=\leray
\Theta^2(-\Delta)$, but once it is satisfied for $Q=\leray \Theta^2(-\Delta)$, for any $\tilde{\Theta}>\Theta$, $\tilde Q=\leray \tilde \Theta^2(-\Delta)$ also satisfies the uniformly elliptic condition. If $Q^{(N)}$ undergoes a proper scaling (see Definition \ref{def proper scaling}), then by Condition \eqref{condition-diagonal}, $A^{(N)}\geq \frac{1}{2}\identity$, for all sufficiently large $N$.

\subsubsection{Classical solutions}
First, let us have a look at the classical point-wise solutions.
\begin{defn}[Classical Solutions]
	A predictable process $u_t (x)\in C ([0,T]\times\manifold)$ is called a classical solution to the stochastic transport equation if for any $x\in\manifold$, $u (x)$ and the first order derivatives of $u$ at $x$ are semimartingales and the following equation holds in any local coordinate system:
	\begin{equation}
	\der u_t (x) = \partial_i u_t (x) \circ \der W_t^i (x),
	\end{equation}
	where the RHS can also be written in a coordinate-free way as $\circ \der W_t (x) \cdot \gradient u_t (x)$.
\end{defn}

For a given smooth initial value $u_0$ there is a solution given by $u_t (x)=u_0 (X_t^{-1} (x))$, where $X_t (x)$ is the stochastic flow given the following Cauchy problem 
\begin{equation}\label{eq-stoch-flow}
\begin{cases}
\der X_t (x)=-\circ \der W_t (X_t (x)) 
\\ X_0 (x)=x
\end{cases},
\end{equation}

and by \cite{kunita1990stochastic}, $X^{-1}_t$ and $X_t$ are smooth diffeomorphisms and their spatial derivatives of all orders are continuous in $t$.
Sometimes it is more convenient to write the equation in It\^{o} form. By \cite{kunita1990stochastic}, the It\^{o}-Stratonovich corrector is 
\begin{equation}
\frac{1}{2} \sum_{k=1}^{\infty}  A^{ij}\partial_i\partial_ju=\frac{1}{2}Lu,
\end{equation} 
and the It\^o form of the equation is
\begin{equation}\label{eq: ito form}
    \der u= \nabla u \cdot \der W+\frac{1}{2}Lu.
\end{equation}
There is another derivation in \cite{flandoli2022eddy} using a decomposition of the noise.

\subsubsection{Solutions with irregular initial data}
In addition to smooth initial data, we also want to treat the equation when the initial value is not regular enough. One example is the white noise solution which is in $H^{-d/2-} (\manifold)$ (Proposition \ref{prop regularit of wn}). Therefore, we need a weak formulation in analytic sense of the stochastic transport equation or define it as a stochastic differential equation in a Sobolev space with negative regularity.  We use the It\^o form as it is easier to deal with.

We formulate the equation in terms of stochastic differential equation in the Sobolev space $H^{\alpha-2}$. 

\begin{defn}
    Suppose $\alpha\in \mathbb{R}$ and $T>0.$ We call $u$ a solution (in $H^\alpha(\manifold)$) to the stochastic transport equation if it is $\mathcal{F}_t$-adapted and  $u\in C([0,T],H^\alpha(\manifold))$, such that
\begin{equation}
    u_t=u_0+\int_0^t \nabla u_s \cdot \der W_s +\frac{1}{2}\int _0^t Lu_s\der s,
\end{equation}
for all $t\in[0,T]$.
\end{defn}

The multiplication operator $\nabla u_s\cdot $, which we also denote by $\multiplyfunc{\nabla u_s\cdot}$ to avoid ambiguity, is a bounded operator from $H^{\beta}(\manifold,\tanbdl)$ to $H^{\alpha-1}(\manifold)$ defined by $X \mapsto\nabla u_s \cdot X$, for $\beta>|\alpha -1|+d/2$. By \cite{bahouri2011fourier}, $\hnormm{X \cdot Y}{-\alpha-1}\lesssim\|X\|_{H^{-\alpha-1} }\|Y\|_{H^{\beta}}$ for smooth vector fields $X,Y$, and then the multiplication operator $\multiplyfunc{\nabla u_s\cdot}$ can be extended to a bounded operator from $H^{\beta} $ to $H^{\alpha-1}$.

The $L_0^2$ norm of the operator is bounded by $\|\nabla u_s\|_{H^{\alpha-1}}|\tr_{H^{\beta}}(Q)|^{1/2}\lesssim\|u_s\|_{H^\alpha}$. By continuity of $u$ in $H^\alpha$, $\int_{0}^t \tr_{H^{\alpha-1}}(\multiplyfunc{\nabla u_s\cdot}Q\multiplyfunc{\nabla u_s\cdot })\lesssim \int_0^t\hnormm{u_s}{\alpha}^2<\infty.$ Therefore, the stochastic integral on the RHS is well defined in $H^{\alpha-1}\subset H^{\alpha-2}$. The drift term is also well defined in $H^{\alpha-2}$ as $L$ is a bounded operator from $H^\alpha$ to $H^{\alpha-2}$.

Since $\nabla \cdot W=0$, $\nabla u_s \cdot \der W_s=\nabla \cdot (u_s\der W_s)$. By testing against smooth functions, the above definition is equivalent to the following weak formulation:

\begin{prop}\label{defn-weak-ito}
	 Let $\alpha\in\mathbb{R}$, $[0,T]$ be an interval.   Suppose $\alpha\in \mathbb{R}$, $T>0$ and  $u\in C([0,T],H^\alpha(\manifold))$ is an $\mathcal{F}_t$-adapted process. $u_t$ is  a  solution to the stochastic transport equation with the driving noise $W_t$ if and only if for any smooth function $\phi \in C^\infty (\manifold)$, the process $ (u_t,\phi)$ satisfies the following:
	 \begin{align}
	  (u_t,\phi)- (u_0,\phi)=&-\int_{0}^{t}  (u_s, \gradient \phi\cdot\der W_s )+\int_0^t (u_s,\frac{1}{2}L\phi)\der s.
	 \end{align}
\end{prop}

Later we also need to use the explicit formula for the quadratic variation of $ (u,\phi)$:
\begin{equation}
\der \quavar{ (u,\phi)}_t=  (u_t\nabla\phi,Q (u_t\nabla\phi))\der t,
\end{equation}

And with the symmetry of $Q$, we get that for any functions $\phi,\varphi\in C^\infty (\manifold)$,
the covariation of the two processes $ (u,\phi)$ and $ (u,\varphi)$ is given by
\begin{equation}\label{quav-test}
\der \quavar{ (u,\phi), (u,\varphi)}_t= (u_t\nabla\phi, Q  (u_t \nabla \varphi))\der t.
\end{equation}

Now we introduce the following mild formulation of the stochastic transport equation. 
For $0\leq s<t$,
\begin{equation}\label{eq-mild}
u_t=P_{t-s}u_s + \int_{s}^{t}P_{t-r}\multiplyfunc{\nabla u_r\cdot}\der W_r=P_{t-s} u_s +\int_s^t P_{t-r}(-\nabla^\ast)\multiplyfunc{u_r}\der W_r,
\end{equation} 
where $P_t=e^{\frac{1}{2}tL}$ is the semigroup generated by $\frac{1}{2}L$. By \cite{berline1992heat},\cite{taylor2011pseudodifferential}, $e^{tL}$ exists and has a smooth integral kernel (also smooth in $t$) for $t>0$, and it is a $C_0$-semigroup on $H^\alpha$ for all $\alpha\in\mathbb{R}$. 

By \cite{da2014stochastic}, a solution is also a mild solution.
\begin{prop}
	Let $u$ be a solution to the stochastic transport equation in the sense of definition \ref{defn-weak-ito}. Then $u$ is also a mild solution, i.e., for $0\leq s<t<T$, $\prob$-a.s.,
	(\ref{eq-mild}) holds.
\end{prop}

\subsection{Well-posedness of stochastic transport equations}
We will show that the solution to the stochastic transport equation is given by $u_t= (X_t^{-1})^*u_0$, where $X_t$ is the stochastic flow given by (\ref{eq-stoch-flow}), so we would like to discuss the properties of the stochastic flow first. 

The stochastic equation on the manifold should be understood as for any smooth function $\phi\in{C^\infty (\manifold)}$,
\begin{equation}\label{eq-st-test}
\der \phi (X_t)=-\nabla \phi (X_t)\cdot \circ \der W_t (X_t),
\end{equation}
and equivalently the equation holds locally in any coordinate system. It is sufficient to check the equation only for a family of charts that cover the manifold due to the coordinate-free nature of Stratonovich equations (as chain rule still holds). By \cite{kunita1986lectures} Theorem 1.4.1 we know that the equation (\ref{eq-stoch-flow}) determines a stochastic flow of diffeomorphism and (\ref{eq-st-test}) holds. (\ref{eq-st-test}) also holds as $H^n$-valued semimartingales ($\phi\circ X_t$ viewed as functions on $\manifold$)  for any $\phi\in{C^\infty ({\manifold})}$ and $n\in\mathbb{N}$. For any $t\leq0$, $X_t$ is a $C^\infty$ diffeomorphism almost surely, and by approximating with linear interpolation or mollification of the noise  (Wong Zakai's theorem), we can know that $X_t$ is measure-preserving.

For any smooth measure preserving diffeomorphism $f$ of the manifold, the pullback of a distribution $v$ by $f$ is defined as $ (f^*v,\phi)= (v,\phi\circ f^{-1})$ for any smooth function $\phi$.

Now we show that $u_t= (X_t^{-1})^*u_0$ is a  solution to the stochastic transport equation.

\begin{thm}
	Suppose $\alpha\in\mathbb{R}$. For given initial value $u_0$ that is a $\mathcal{F}_0$-measurable, $H^\alpha$-valued random variable, $u_t= (X_t^{-1})^* u_0$ is a  solution to the stochastic transport equation, where $X_t (x)$ is the stochastic flow generated by the driving noise.
\end{thm}
\begin{proof}
	Note that $\frac{1}{2}L$ is the generator of the semigroup related to the Markovian process $X_t (x)$,  
	\begin{equation}
	\der\phi (X_t)=- (\nabla \phi) (X_t)\cdot \der W(X_t)+\frac{1}{2} (L\phi) (X_t).
	\end{equation}
	
	By the properties of the stochastic flow (see \cite{kunita1990stochastic}), the process $\varphi_t=X_t^*\phi=\phi\circ X_t$ is a semimartingale in $H^n$ for all $n\in\mathbb{N}$ and satisfies the following SDE:
	\begin{equation}
	\der \varphi_t(x) =\nabla\phi(X_t(x))\cdot\der W_t(X_t)+ \frac{1}{2}(L\phi)(X_t(x))\der t.
	\end{equation}
   We can choose $n>-\alpha$. Then
	\begin{equation}
	 (u_0,\der\varphi_t)=- (u_0,X_t^* (\nabla \phi\cdot\der W_t)) +  (u_0,\frac{1}{2}X_t^* (L\phi))\der t.
	\end{equation}  
	If $u_t= (X_t^{-1})^* u_0$, then $ (u_t,\phi)= (u_0,\varphi_t)$, since $X_t$ preserves the volume on the Riemannian manifold.
	\begin{align*}
	 (u_t,\phi)- (u_0,\phi)=& (u_0,\varphi_t-\varphi_0)
	\\=&\int_{0}^{t}\left ( -(u_0,X_s^* (\nabla \phi\cdot \der W_s))+  (u_0,\frac{1}{2}X_s^* (L\phi))\der s \right)
	\\=&\int_{0}^{t}\left ( - (u_s,\nabla \phi\cdot\der W_s)+  (u_s,\frac{1}{2}L\phi)\der s \right).
	\end{align*}
	
	Therefore, $u_t= (X_t^{-1})^*u_0$ is a  solution to the stochastic transport equation.
    The continuity of $u_t$ in $H^\alpha(\manifold)$ follows from the fact that the $k$-th order derivatives of the stochastic flow $X_t$ and its inverse are continuous in $t$ for any $k\in \mathbb{N}$ (see \cite{kunita1990stochastic}).
	
\end{proof}

Then we show that the  solution we construct  is the unique solution in $H^\alpha$. To show uniqueness, we can assume $\alpha\leq 0$ without loss of generality.
\begin{thm}
If $u_0$ is a $\mathcal{F}_0$-measurable $H^\alpha$-valued random variable, and $u$ is a  solution in $H^\alpha$ to the stochastic transport equation, for some $\alpha\leq 0$, then $u_t= (X_t^{-1})^*u_0$.  
\end{thm}
\begin{proof}
	We will show that for any smooth solution, its pairing with $u_t$ does not depend on time. 
	
	Let $\phi$ be a smooth function on $\manifold$. By the previous discussion we know that $\phi_t := (X_t^{-1})^*\phi$ is a  solution in $H^\beta$ for any $\beta\in \mathbb{R}$. Take $\beta>-\alpha+4$.
	
    Using It\^o's formula (Proposition \ref{prop ito's formula}) for a function $F(u,v)=(u,v)_{H^{\alpha-2},H^{2-\alpha}}\in C^2(H^{\alpha-2}\times H^{\beta-2})$ , $D^2F=\begin{pmatrix}
        0 & \iota_{H^{\beta-2}\rightarrow (H^{\alpha-2})^\ast} 
        \\ \iota_{H^{\alpha-2}\rightarrow(H^{\beta-2})^\ast}& 0
    \end{pmatrix},$
    where $\iota_{H^{\beta-2}\rightarrow (H^{\alpha-2})^\ast}$ is the natural embedding from $H^{\beta-2}$ to $(H^{\alpha-2})^\ast =H^{2-\alpha}$ and the other $\iota$
	\begin{align*}
	&\der (u_t,\phi_t)_{H^{\alpha-2},H^{2-\alpha}}
	\\=& (u_t,\nabla\phi_t\cdot \der W_t)+\frac{1}{2} (u_t,L\phi_t)\der t+
    (\nabla u_t \cdot \der W_t,\phi_t)+\frac{1}{2}(Lu_t,\phi_t)\der t
    \\+&\frac{1}{2}\tr\begin{pmatrix}
        \multiplyfunc{\nabla\phi_t\cdot}Q\multiplyfunc{\nabla u_t\cdot} &
        \\& \multiplyfunc{\nabla u_t\cdot}Q\multiplyfunc{\nabla \phi_t\cdot}
    \end{pmatrix}\der t
	\\=& (\nabla(u_t\phi_t),\der W_t)+ (u_t,L\phi_t)\der t+ (\nabla u_t, A\nabla\phi_t)\der t
	\\=&  (u_t,\nabla \cdot  (A\nabla \phi_t))\der t- (u_t,\nabla\cdot  (A\nabla \phi_t))\der t
	\\=&0.
	\end{align*}
In the second step, the upper left block is an operator on $(H^{\alpha-2})^\ast$, and the lower block is an operator on $(H^{\beta-2})^\ast$. 
    
	Therefore, almost surely
	\begin{equation}\label{eq-conservation}
	 (u_t,\phi_t)= (u_0,\phi)= ( (X_t^{-1})^*u_0, \phi_t). 
	\end{equation}

	Let $\phi$ run over a countable dense subset of $C^\infty (\manifold)$. Then almost surely $\phi_t$ is also a dense subset ($ (X_t^{-1})^*$ is a continuous linear automorphism of $C^\infty (\manifold)$ due to the fact that $X_t$ is a diffeomorphism almost surely). By (\ref{eq-conservation}), $u_t= (X_t^{-1})^* u_0$. Therefore, for any given initial value that is almost surely in $H^\alpha$, the  solution is unique. 
	
\end{proof}

\subsection{\texorpdfstring{$L^2$}{L2} solutions, white noise solutions and their regularities  }
Our main focus is on two specific types of initial data, namely the deterministic $L^2$ case and white noise case.

Suppose $u_0$ is in $L^2 (\manifold)$. Then the solution $u$ to the stochastic transport equation satisfies $\lVert u_t \rVert_{L^2 (\manifold)}=\lVert u_t \rVert_{L^2 (\manifold)}$ $\prob$-almost surely, since the stochastic flow $X_t$ is measure preserving. 

Now we show the law of white noise is a stationary distribution of the stochastic transport equation.
\begin{thm}
	Suppose that $u_0$ is $\mathcal{F}_0$-measurable white noise. Then the solution to the stochastic transport equation $u_t= (X_t^{-1})^* u_0$ is stationary, i.e. $u_t$ is also a white noise for $t\geq 0$.
\end{thm}
\begin{proof}
	By Proposition \ref{prop-char-white}, it suffices to show that for any smooth function $\phi$ on $\manifold$, 
	\begin{equation}
	\Expectation[\exp (i (u_t,\phi))]=\exp\left (-\frac{1}{2}\LtwoMnorm{\phi}^2\right).
	\end{equation}
	Note that $u_0$ and $X_t$ are independent since the noise is independent of $\mathcal{F}_0$, and $X_t$ is measure preserving. Then we have 
	\begin{align*}
	\Expectation[\exp (i (u_t,\phi))]=& \Expectation[\exp (i (u_0,\phi\circ X_t))]
	\\ = & \Expectation\big[ \Expectation\left[ \exp (i (u_0,\phi\circ X_t))\arrowvert X_t \right] \big]
	\\  = &\Expectation[\exp (-\frac{1}{2}\LtwoMnorm{\phi\circ X_t}^2)]
	\\ = &\exp\left (-\frac{1}{2}\LtwoMnorm{\phi}^2\right).
	\end{align*}
	
\end{proof}

Then we discuss regularity estimates of $L^2$ and white noise solutions. Firstly, we can bound the moments of the norm of the increments and then use Kolmogorov continuity criterion to get the H\"older regularity in time.

We look at $L^2$ solutions first.
\begin{thm}
	\label{thm l2 h-1 bd}Suppose $T>0$, $p\in[2,\infty)$, and $u$ is a $L^2$ solution with deterministic initial value $u_0$.
	Then there exists a constant $C>0$ which may depend on the manifold and $p,A$ such that for any $t,s\in[0,T]$,  
        \begin{equation}\label{eq bd h-1}
	\lpnormprob{\hnormm{u_t-u_s}{-1}}{p} \leq C |t-s|^{1/2} \LtwoMnorm{u_0}.
	\end{equation}
    Moreover, the constant $C$ can be chosen uniformly in $A$ provided $A$ and $A^{-1}$ are uniformly bounded in $L^\infty$.
\end{thm}
\begin{proof}
 By mild formulation,
	$$u_t-u_s= (P_{t-s}-\identity)u_s +\int_{s}^{t}P_{t-r}\nabla\cdot(u_r\der W_r).$$
	The first term can be easily bounded by heat kernel estimates (can be easily proved using the equivalence of $\anorm{\cdot}{-1}$ and $H^{-1}$ and the spectral decomposition with respect to $L$) and the $L^2$ boundedness of $u_s$:
    \begin{equation}\label{u -1 bd 1}
        \hnormm{ (P_{t-s}-\identity)u_s}{-1}\lesssim_{A,\manifold} (t-s)^{1/2}\LtwoMnorm{u_s}= (t-s)^{1/2}\LtwoMnorm{u_0}.
    \end{equation}

	Then we bound the stochastic convolution using  Proposition  \ref{prop moment estimate via quadratic} and Corollary \ref{cor qaalpha expansion},
    \begin{align*}
        &\lpnormprob{\anorm{\int_s^t P_{t-r} \nabla u_r \cdot \der W_r}{-1}}{p}^{2}
        \\& \lesssim_p\int_s^t \left(\Expectation \left(\int_{\manifold\times\manifold} Q(x,y)\cdot u_r(x)u_r(y)\nabla_x\nabla_y q_{t-r}^{A,-1}(x,y) \der x \der y\right)^{p/2} \right)^{2/p} \der r
        \\ &= \int_s^t \lpnormprob{\tr ((\multiplyfunc{u_r}Q\multiplyfunc{u_r})\circ (\nabla(1-L)^{-1}P_{2(t-s)}\nabla^\ast)}{p/2}\der r.
    \end{align*}
    By Proposition  \ref{prop trace uqu}, we have
    \begin{align*}
        \tr (\multiplyfunc{u_r}Q\multiplyfunc{u_r})&=\int_\manifold u_r(x)^2\tr A(x)\der x
        \\&\lesssim_\manifold  \LtwoMnorm{u_r}^2\lpmani{A}{\infty}.
    \end{align*}
    The operator on $L^2(\manifold,\tanbdl)$, $$\nabla(1-L)^{-1}P_{2(t-s)}\nabla^\ast,$$ is bounded uniformly in $A$, provided that $A$ and $A^{-1}$ are uniformly bounded in $L^\infty$, as we have $$\nabla\in L(H^1(\manifold,\tanbdl),L^2(\manifold)),$$ $$\|(1-L)^{-1}e^{tL}\|_{H^{-1}_A\rightarrow H_A^1}\leq 1,$$  and $$\nabla^\ast\in L(L^2(\manifold),H^{-1}(\manifold,\tanbdl)).$$
Note that   $(1+\|A^{-1}\|_{L^\infty})^{-1/2}\hnormm{\cdot}{1}\leq\anorm{\cdot}{1}\leq (1+\|A\|_{L^\infty})^{1/2}\hnormm{\cdot}{1}) $. Then by Proposition \ref{prop trace(AB) < opnorm A trace B},
    \begin{equation*}
        \tr ((\multiplyfunc{u_r}Q\multiplyfunc{u_r})\circ (\nabla(1-A)^{-1}P_{2(t-s)}\nabla^\ast)\lesssim_{A,\manifold} \LtwoMnorm{u_0}^2.
    \end{equation*}
    Hence, 
    \begin{equation} \label{u -1 bd 2}
        \lpnormprob{\anorm{\int_s^t P_{t-r} \nabla u_r \cdot \der W_r}{-1}}{p} \lesssim_{A,\manifold,p} \LtwoMnorm{u_0}(t-s)^{1/2}.
    \end{equation}
    Combining the estimates \eqref{u -1 bd 1} and \eqref{u -1 bd 2} and the fact that the $H^1_A$ and $H^1$ are equivalent (up to a constant that can be bounded by some functions of $\|A\|_{L^\infty}$ and $\|A^{-1}\|_{L^\infty}$), we get the desired bound.
    \end{proof}

Note that $\LtwoMnorm{u_t-u_0}\leq 2\LtwoMnorm{u_0}$, and then by interpolation with (\ref{eq bd h-1}), we get
\begin{equation}
    \lpnormprob{\hnormm{u_t-u_s}{-\kappa}}{p}\lesssim  (t-s)^{\kappa/2}\LtwoMnorm{u_0},
\end{equation}
for any $\kappa\in[0,1],p\in[1,\infty).$ Then by Kolmogorov continuity criterion, we have the following corollary:
\begin{cor}\label{cor-holder-moment-l2}
	Let $T>0$ and $\{u_t\}_{0\leq t \leq T}$ be a $L^2$ solution to the stochastic transport equation. Then for all $\alpha \in [0,1/2),\varepsilon>0$, $u$ is almost surely in the space $C^{\alpha} ([0,T],H^{-2\alpha-\varepsilon} (\manifold))$.
	
	Moreover, for any $p\geq 1$,
	\begin{equation}
	\Expectation \left\lVert u \right\rVert_{C^{\alpha} ([0,T],H^{-2\alpha-\varepsilon})}^p \leq C_{\alpha,\varepsilon,\manifold,p,A,T}\LtwoMnorm{u_0}
	\end{equation}
	for some constant $C_{\alpha,\varepsilon,\manifold,p,A,T}<\infty$, which can be chosen uniformly in $A$ if $A$ and $A^{-1}$ are uniformly bounded in $L^\infty$.
\end{cor}

Now we look at the regularity of white noise solutions.
\begin{thm}\label{thm-regularity-moment}
	Let $T>0$ and $\{u_t\}_{0\leq t \leq T}$ be a white noise solution to the stochastic transport equation. Then for any $\kappa\in  (0,1), p\in[2,+\infty)$, there exists a constant $C=C(p,\kappa,T,A,\manifold)>0$ such that
	\begin{equation}\label{bd wn regularity moment}
	\left (\Expectation\hnormm{u_t-u_s}{-\frac{d}{2}-\kappa}^p\right)^{\frac{1}{p}}\leq C_{p,\kappa,T,A,\manifold}  (t-s)^{\kappa/2}. 
	\end{equation}
    Moreover, the constant $C$ can be chosen uniformly in $A$, provided that $A$ and $A^{-1}$ are uniformly bounded in the $C^k$ norm for some sufficiently large integer $k$. 
\end{thm}

\begin{proof}
    By mild formulation, 
    \begin{equation}
        u_t-u_s=(P_{t-s}-1)u_s + \int_{s}^{t} P_{t-r} \nabla u_r \cdot \der W_r.
    \end{equation}
    By  Proposition \ref{prop-wn-moment},  
    \begin{align*}
       & \lpnormprob{\anorm{(P_{t-s}-1)u_s }{-d/2-\kappa}}{p}
        \\&\lesssim_{p} \left(\tr((1-L)^{-d/2-\kappa}(P_{t-s}-1)^2)\right)^{1/2}  
    \end{align*}
   
    Then we use Lemma \ref{lma: heat kernel difference trace bound} to get
    \begin{equation}\label{bd wn regularity 1}
        \lpnormprob{\anorm{(P_{t-s}-1)u_s }{-d/2-\kappa}}{p}\lesssim_{A,T,p,\kappa,\manifold} (t-s)^{\kappa/2},
    \end{equation}where the implicit constant can be chosen uniformly in $A$, provided that the $C^k$ norm of $A,A^{-1}$ for some large enough $k$.
    
    We bound the stochastic convolution using  Proposition  \ref{prop moment estimate via quadratic} and Corollary \ref{cor qaalpha expansion} by
    \begin{align*}
        &\lpnormprob{\anorm{\int_s^t P_{t-r} \nabla u_r \cdot \der W_r}{-d/2-\kappa}}{p}^{2}
        \\  &\lesssim_p \int_s^t \lpnormprob{\tr_{H_A^{-d/2-\kappa}}(P_{t-r}\nabla ^\ast \multiplyfunc{u_r} Q  \multiplyfunc{u_r}  \nabla P_{t-r})}{p/2} \der r
        \\&=
        \int_s^t \left(\Expectation \left(\int_{\manifold\times\manifold} u_r(x)u_r(y)Q(x,y)\cdot \nabla_x\nabla_y q_{t-r}^{A,-d/2-\kappa}(x,y) \der x \der y\right)^{p/2} \right)^{2/p} \der r,
    \end{align*}
    where $q_t^{A,\alpha}$ is the integral kernel of the operator $(1-L)^\alpha e^{tL}$.
    By  Proposition \ref{prop-est-wn-quadratic}, and applying the bound in Lemma \ref{lma:heat kernel estimate C2}, we get
    \begin{align*}
        &\left(\Expectation \left(\int_{\manifold\times\manifold} u_r(x)u_r(y)Q(x,y)\cdot \nabla_x\nabla_y q_{ t-r}^{A,-d/2-\kappa}(x,y) \der x \der y\right)^{p/2}\right)^{2/p}
      \\  &\lesssim_p \int_\manifold Q(x,x)\cdot \nabla_x\nabla_y q_{ t-r}^{A,-d/2-\kappa}(x,y)|_{y=x} \der x
     \\& \leq |\manifold |\|A\|_{L^\infty} \|q^{A,-d/2-\kappa}_{ t-r}\|_{C^2}
     \\ &\lesssim_{\manifold,A,T} (t-r)^{-1+\kappa}.
    \end{align*}
    The constant (coming from Lemma \ref{lma:heat kernel estimate C2}) may depend on $C^k$ norm of $A$ and $A^{-1}$ for some large enough $k$.
    Therefore, 
       \begin{align*}
        &\phantom{=}\lpnormprob{\anorm{\int_s^t P_{t-r} \nabla u_r \cdot \der W_r}{-d/2-\kappa}}{p}
        \\& \lesssim_{{A,T,p,\kappa,\manifold}}\left( \int_s^t (t-r)^{-1+\kappa} \der r \right)^{1/2}
        \\&\lesssim_\kappa(t-s)^{\kappa/2}.
        \end{align*}
        Together with the bound \eqref{bd wn regularity 1}, we get \eqref{bd wn regularity moment}, with $H^\alpha$ norm replaced by its equivalent norm $H^\alpha_A$, and the constant is uniformly bounded in $A$ if $A$ and $A^{-1}$ are uniformly bounded in the $C^k$ norm for sufficiently large $k$. 
        \end{proof}

Then by applying the Kolmogorov continuity criterion, we have the following corollary showing that the $C^\alpha$ norm  (in time variable) of the solutions as processes in $H^{-\frac{d}{2}-2\alpha-\varepsilon}$ has bounded moments:

\begin{cor}\label{cor-holder-moment}
	Let $T>0$ and $\{u_t\}_{0\leq t \leq T}$ be a white noise solution to the stochastic transport equation. Then for all $\alpha \in [0,1/2),\varepsilon>0$, $u$ is almost surely in the space $C^{\alpha} ([0,T],H^{-\frac{d}{2}-2\alpha-\varepsilon} (\manifold))$.
	
	Moreover, for any $p>1$,
	\begin{equation}
	\Expectation \left\lVert u \right\rVert_{C^{\alpha} ([0,T],H^{-\frac{d}{2}-2\alpha-\varepsilon})}^p \leq C_{\alpha,\varepsilon,\manifold,p,A,T}
	\end{equation}
	for some constant $C_{\alpha,\varepsilon,\manifold,p,A,T}<\infty$, which can be chosen uniformly in $A$, provided that $A$ and $A^{-1}$ are uniformly bounded in the $C^k$ norm for some sufficiently large integer $k$. 
\end{cor}
The corollary can be used to prove tightness in $C^{\alpha-\varepsilon} ([0,T],H^{-d/2-2\alpha})$ since $C^\alpha  ([0,T], H^\beta)$ is compactly embedded in $C^{\alpha-\varepsilon } ([0,T], H^{\beta-\varepsilon})$ for any $\varepsilon>0$. 

	\section{Scaling limits}\label{section-scaling-limit}

    In this section we will show that if we scale the noise properly, the solutions to the stochastic transport equations converge to some universal objects depending on the initial conditions. 
    
    If the initial value is in $L^2$, the $L^2$ norm is always preserved by the stochastic transport equation ($\LtwoMnorm{u_t}=\LtwoMnorm{u_0}$) while the limiting process is the solution to a dissipative PDE with $\LtwoMnorm{u^{(\infty)}_t}<\LtwoMnorm{u_0}$. Therefore the convergence cannot happen in $L^2(\manifold)$ even in probability. We will get a quantitative estimate on the convergence rate, which is of the same order as in \cite{flandoli2024quantitative}.
    
    If we consider the white noise solutions, the limiting process is given by a stochastic heat equation with additive noise of the form of $ (-\Delta)^{-1/2}\xi$, where $\xi$ is the space time white noise on the manifold $\manifold$. 

We consider a sequence of solutions $u^{ (N)}$ of the stochastic transport equation driven by noises $W^{ (N)}$ with covariance $Q^{ (N)}$. And we let $A^{ (N)} (x)=Q^{ (N)} (x,x)$  be the value of the covariation on the diagonal.
\subsection{Assumptions on the scaling of noises}\label{s-proper-scaling}
In order to have convergence to heat equations, we need the convergence of the It\^{o}-Stratonovich corrector. Therefore, we need to assume
\begin{equation}\label{condition-diagonal}
\lim_{n\rightarrow\infty} A^{ (N)} = \identity, \: \textrm{in } \: C^\infty (\manifold, \cotbdl\otimes\tanbdl),
\end{equation}
where $\identity (x)$ is the identity matrix in the space of linear transformations on $T_x\manifold$. 

It would be more convenient to consider the case where $A^{ (N)}$ are fixed, which can be achieved when the manifold has enough symmetry, for example $\mathbb{S}^d, \mathbb{T}^d$ with the usual metric. But in general we do not know if we can keep $A^{ (N)}$ fixed, as it is difficult to keep track of the divergence free property and the covariance on the diagonal at the same time.  

The trace of the covariance operator  (as an operator on $L^2 (\manifold, \tanbdl)$) is given by the integration on the diagonal of the trace of the integral kernel:
\begin{equation*}
\tr Q = \int_{\manifold}  \tr Q (x,x) \der x
\end{equation*}
Therefore, the condition (\ref{condition-diagonal}) guarantees the convergence of $\tr Q^{ (N)}$ to 
$|\manifold|^2 d$, where $|\manifold|$ denotes the volume of the manifold and $d$ is the dimension of the manifold. The boundedness of the trace of $Q$ is also guaranteed by (\ref{condition-diagonal}).

We also want the noise to be decorrelated, which means that the correlation between different points decays fast enough. In fact, it is enough to control the contribution of the correlation off the diagonal if we know that the operator norm of $Q^{ (N)}$ goes to zero, or equivalently, the Hilbert-Schmidt norm of $Q^{ (N)}$ goes to zero as condition (\ref{condition-diagonal}) already guarantees the boundedness of trace.

Therefore, we introduce our second assumption: as $N$ goes to infinity,
\begin{equation}\label{condition-l2}
\lVert Q^{ (N)}\rVert_{L^2\rightarrow L^2} \longrightarrow 0,
\end{equation} 
or equivalently  (under condition (\ref{condition-diagonal})),
\begin{equation}\label{condition-hs}
Q^{ (N)} \longrightarrow 0 \: \text{in }  L^2 (\manifold\times\manifold, \tanbdl \boxtimes\tanbdl), \tag{\ref*{condition-l2}{}'}
\end{equation} which is equivalent to the Hilbert Schmidt norm of the operator $Q^{ (N)}$  (as an operator on $L^2 (\manifold,\tanbdl$)) going to zero, since
\(
    \hsnorm{Q^{ (N)}}
    =\|Q\|_{L^2 (\manifold\times\manifold, \tanbdl\boxtimes \tanbdl )}
\).

The equivalence of condition (\ref{condition-l2}) and condition (\ref{condition-hs}) under condition (\ref{condition-diagonal}) is guaranteed by the fact that the trace of operator $Q^{ (N)}$ is bounded and that $$\opnorm{ Q^{(N)}}\leq \hsnorm{Q^{(N)}}\leq \opnorm{Q^{(N)}}^{1/2}(\tr Q^{(N)})^{1/2}.$$ (It follows from $\|\theta\|_{l^\infty}\leq \|\theta\|_{l^2}\leq \|\theta\|_{l^\infty}^{1/2}\| \theta\|_{l^1}^{1/2}$, where $\theta$ is the sequence of eigenvalues of $Q^{(N)}$.) 

\begin{defn}
\label{def proper scaling}
    Let $\{Q^{(N)}\}_N$  be a sequence of smooth, positive semidefinite and divergence-free
covariance operators. We say that $\{Q^{(N)}\}_N$ undergoes \emph{a proper scaling}, if (\ref{condition-diagonal})
 and (\ref{condition-l2}) are satisfied, or equivalently (\ref{condition-diagonal}) and (\ref{condition-hs}
) are satisfied.
\end{defn}

    
\begin{rmk}
    In condition (\ref{condition-diagonal}), we do not necessarily require convergence in \( C^\infty \). It is sufficient for \( A^{(N)} \) to converge in \( C^k \) for some sufficiently large \( k \). This requirement arises from two key factors:
    
    1. The equivalence between the  Sobolev spaces \( H^{-d/2-\kappa}_{A^{(N)}} \) and \( H^{-d/2-\kappa} \) for \( \kappa \in [0,1) \), which depends on uniform control over the norms of \( A \) and \( A^{-1} \).
    
    2. The implicit constants appearing in the heat kernel estimates, which may require uniform boundedness of \( C^k \)-norms of \( A \) and \( A^{-1} \) for some sufficiently large \( k \).

    To ensure uniform equivalence of Sobolev norms, one typically needs \( k \geq d/2 \). For the heat kernel estimates, we expect a similar requirement, but the precise regularity condition on \( A \) is not explicitly tracked in the geometry or analysis literature (see \cite{berline1992heat}, \cite{taylor2011pseudodifferential}).

    Nevertheless, in the  proper scaling we construct, \( A^{(N)} \) is chosen to converge in \( C^\infty \), ensuring that all estimates remain valid without ambiguity.
\end{rmk}

Roughly speaking, we want to make the covariance $Q^{ (N)} (x,y)$  converge to $\identity (x)1_{y} (x)$. Then the trajectories of the flow starting from different points would be like independent Brownian motions on the manifold.
Although we cannot get independence because the flow has to be incompressible, we expect that in the limit they should be  asymptotically independent. 
If we look at $n$ different particles, whose initial positions are independently uniformly distributed, in the stochastic flow driven by $\circ\der W^{ (N)}$, then the joint law of their trajectories becomes independent Brownian motions in the limit.  

\subsection{Examples of proper scalings}
In this subsection, we will give some examples of proper scalings of the noises.

On manifolds with enough symmetries, we can easily tune the covariances of the noises to be constant on the diagonal, which gives even better condition than (\ref{condition-diagonal}). On general manifolds, we do not know whether if we can still keep $A^{(N)}$ fixed, but we can construct a proper scaling with renormalized mollified white noise  satisfying condition (\ref{condition-diagonal}) and (\ref{condition-l2}).

In this discussion we take the ``isotropic'' noises $\der W^{ (N)}=\leray\Theta^{ (N)} (-\hodgelplc)\xi$, where $\leray$ is the Leray projection, $\Theta^{ (N)}$ are functions on the spectrum of $-\hodgelplc$ (minus Hodge Laplacian) and $\xi$ is space-time white noise on the tangent bundle. We let $\theta_j=\Theta (\lambda_j)$, where $\lambda_j$ are the eigenvalues of $-\hodgelplc$. In this case $Q^{ (N)}= \left (\leray\Theta^{ (N)} (-\hodgelplc)\right)^*\leray\Theta^{ (N)} (-\hodgelplc)=\leray \left (\Theta^{ (N)}\right)^2 (-\hodgelplc)$.

For simplicity, we omit the upper index $N$ if it holds for any $N\in \mathbb{N}^*$.
If $f$ is an isometry, then the pushforward of $f_*A$ is exactly $A$. In coordinates,
\[A^{ij} (f (x))=A^{kl} (x)\frac{\partial f^i}{\partial x^k} (x)\frac{\partial f^j}{\partial x^l}  (x)\]

\subsubsection{Proper scalings on \texorpdfstring{$\mathbb{T}^d,\mathbb{S}^d$}{Td, Sd}}\label{torusscal}
Now we take $\mathbb{T}^d,\mathbb{S}^d$ as examples to discuss the situation on ``nice" manifolds. 

We start with the torus $\mathbb{T}^d=\mathbb{R}^d/\mathbb{Z}^d$ which is the same case as in the previous works (\cite{galeati2020convergence, flandoli2024quantitative}). In fact, our conditions are almost the same as those in the previous works (they only differ up to  some multiplicative constants).

We take the usual coordinates on the torus which is given by the lift of the map $\mathbb{R}^d\rightarrow\mathbb{Z}^d$. By translational invariance, we get that $A ^{(N)}(x)=Q^{(N)} (x,x)$ does not depend on the position of $x$. Since the maps swapping two coordinates are isometries, $A^{ii}=A^{jj}$. Since reflecting the $i$-th coordinate is also an isometry, $A^{ij}=-A^{ij}=0$ for all $j\neq i$. Therefore, $A$ is a constant times the identity matrix, where the constant is the trace of $\tr Q^{(N)}/d$, since 
\[\tr Q^{(N)}=\int_{\manifold}\tr Q^{(N)} (x,x)\der x= A^{ii}d. \]

Now condition (\ref{condition-diagonal}) is equivalent to $\lVert\theta^{ (N)}\rVert_{l^2}$ (the $l^2$ norm) going to $\sqrt{d}$. If we keep the $l^2$ norm of $\theta^{ (N)}$ to be exactly $\sqrt{d}$ we can even get a better condition with $A^{ (N)}=\identity$. Condition (\ref{condition-l2}) is equivalent to $\lVert \theta^{ (N)} \rVert_{l^\infty}$ going to zero. 

The same trick also works when $\manifold$ is the $d$-dimensional sphere $\mathbb{S}^d$. Due to rotational invariance of the metric, we can obtain that in normal coordinates $A^i_j (0)$ commutes with any orthogonal matrix. Therefore, $A^i_j$ must be a constant times the identity matrix at the origin of that normal coordinate map. Since the isometry group ($O (n)$) acts transitively on $\mathbb{S}^d$, the constant does not depend on the position. Thus, $A^i_j$ is  a constant times the identity  matrix. And the trace of $Q$ is the constant times $d$ times the volume of $\mathbb{S}^d$. Therefore, condition (\ref{condition-diagonal}) is equivalent to $\lVert \theta^{ (N)} \rVert_{l^2} \rightarrow  (|\mathbb{S}^d|d)^{1/2}$.

\subsubsection{Proper scalings via mollified white noise}\label{scaling}
As in the last two examples, on ``nice'' manifolds like $\mathbb{T}^d,\mathbb{S}^d$, the covariances on diagonal $A^{ (N)}$ are some constants times the identity  matrix by symmetry arguments. On general manifolds, this property does not hold. Therefore, we cannot  acquire  (\ref{condition-diagonal}) simply by controlling the traces of $Q^{ (N)}$ and it seems impossible to keep $A$ fixed. In this case, we can still construct a sequence of desired noises via mollification of the space-time white noise with the heat kernel. 

\begin{thm}
\label{thm scaling via mollified space time white noise}
    Let $\Theta^h(\lambda)= (4\pi h^2)^{d/4}\sqrt\frac{d}{d-1}(\leray  e^{h^2\lambda/2}) $.
    Then 
    \begin{equation}
        Q^{h}=h^{-d} (4\pi h^2)^{d/2}\frac{d}{d-1}(\leray  e^{h^2\Delta/2})(\leray  e^{h^2\Delta/2})^\ast=(4\pi h^2)^{d/2}\frac{d}{d-1}e^{h^2 \Delta}.
    \end{equation}
    
    As $h\rightarrow0$, $Q^h$ satisfies the proper scaling condition, with the following estimates of convergence rate:
    \begin{itemize}
        \item For any positive integer $l$, \begin{equation}
        \label{estimate Ah}
            \cknorm{A^h-\identity}{l}\lesssim h^2,
        \end{equation} where $A^h(x)=Q^h(x,x)$ is the restriction to the diagonal, and the implicit constant may depend on $\manifold$ and $l$.
\item \begin{equation}\label{estimate Qh op norm}
    \opnorm{Q^h}\leq C h^{d},
\end{equation}
with $C=(4\pi)^{d/2}\frac{d}{d-1}$.
    \end{itemize}
\end{thm}
\begin{proof}
     By Theorem \ref{thm: diagonal leray heat kernel converge},  for $0<h\leq 1$,
     \[
     \cknorm{a_{h^2}-(4\pi h^2)^{-d/2}\frac{d-1}{d}\identity}{l}\lesssim h^{-d+2},
     \] where $a_{h^2}$ is the restriction on the diagonal of the kernel of $\leray e^{h^2\Delta}$. As $A^h=(4\pi h^2)^{d/2}\frac{d}{d-1}a_{h^2}$, we get \eqref{estimate Ah}.
     
     $\Delta$ is positive semidefinite and generates a contraction semigroup on $L^2$, and the Leray projection is an orthogonal projection. Therefore, we get
     \[ \opnorm{Q^h}\leq(4\pi h^2)^{d/2}\frac{d}{d-1}\opnorm{e^{h^2\Delta}}\opnorm{\leray}\leq (4\pi )^{d/2}\frac{d}{d-1} h^{d},\]
     which is the estimate \eqref{estimate Qh op norm}.
\end{proof}

From the heat kernel expansion in Section \ref{subsec heat kernel expansion}. we also get that the typical correlation length scale is of order $h$.

The above scaling of noises falls into the large scale analysis regime that we discussed in the introduction.
We study the stochastic transport equation on rescaled manifold $\lambda\manifold$  (which is  the manifold with a scaled Riemannian metric $ (\manifold,\lambda^2 g)$) driven by space time white noise mollified with kernel $\Theta (-\Delta)$, and then we pull the solution back to $\manifold$ with diffusive scaling ($u^{ (\lambda)}_t (x)=\lambda^{d/2}\tilde u^\lambda_{\lambda^2 t} (\lambda x)$, where the $\lambda x\in \lambda \manifold$ on the RHS is actually  $x\in\manifold$, but we denote it by $\lambda x$ as the spatial distance is scaled by a factor of $\lambda$). Then $u^{ (\lambda)}$ still satisfies a stochastic transport equation, but the noise is mollified with the kernel $\Theta^h (-\Delta)$, with $h=\lambda^{-1}$. For the initial condition, white noise is preserved under the scaling, and $L^2$ initial condition should be rescaled such that  $\tilde u^\lambda_{0} (\lambda x)=\lambda^{d/2} u_0(x)$.


\begin{rmk}
Preliminary computations suggest we can also construct other proper scalings using additional techniques. 

With Theorem \ref{thm: diagonal leray heat kernel converge}, and some fractional calculus techniques, we can also construct proper scaling of the form $Q^h=Ch^{-d-2\alpha}\leray e^{h^2\Delta}(1-\Delta)^\alpha$, for $\alpha>-d/2$ and the noise is given by $\partial_t W^h=h^{-d/2-\alpha}e^{h^2\Delta}(1-\Delta)^{\alpha/2}\xi$. 

We can also take a much larger class of proper scalings with the help of semiclassical analysis in \cite{zworski2022semiclassical}. We can take the noise of the form $h^{d/2+2}\leray \Delta a^{\operatorname{w}}(x,hD)\xi $, where $a(x,p)\in L(T_x\manifold)$ for $(x,p)\in \cotbdl$,  and $a$ is a Schwartz function on the cotangent bundle The semiclassical operator $a^{\operatorname{w}}(x,hD)$ is defined to be the Weyl quantization of $a$.
On $\mathbb{R}^d$, the operator is defined as 
\begin{equation}
    a^{\operatorname{w}}(x,hD) u(x):=\frac{1}{(2\pi h)^d}\int_{\mathbb{R}^d}\int_{\mathbb{R}^d} e^{\frac{i}{h}(x-y,p)} a\left(\frac{x+y}{2},p\right)u(y)\der p\der y.
\end{equation}
For detailed definition of the quantization on manifolds, see \cite{zworski2022semiclassical}. We apply an additional Laplacian to deal with the singularity of  the symbol  of Leray projection at $(x,0)$.
Then the principal symbol of $Q^h$ is $\sigma(Q^h)(x,p)=h^d(|p|^2\identity-p \otimes p)a(x,p)a^\ast(x,p)(|p|^2\identity-p\otimes p)$. $A^h(x)=(2\pi)^{-d}\int_{T_x^\ast\manifold}\sigma(Q^h)(x,p)\der p + O_{C^\infty}(h)$. And possibly with a better choice of the quantization with half densities, the error has a better bound of order $h^2$. The operator norm of $Q^h$ is of order $h^d$.

\end{rmk}

\subsection{Convergence to stochastic heat equations of white noise solutions}\label{s-cvg-wn}
In this subsection, we are going to show that under the proper scaling, the white noise solutions to the stochastic transport equations converge to a limiting object, which is a stochastic heat equation with some additive noise. We will first find the correct limit by computing the generators and then use a tightness argument to show the convergence.
\subsubsection{Convergence of generators}
We are going to define the generators on a family of cylindrical functions $$\cylindricalFunctions :=\left\{F (u)=f\big ( (u,\phi_1),\dots, (u,\phi_n)\big) \Big| n\in\mathbb{N}^*, f\in C_b^2 (\mathbb{R}^n), \phi_j\in C^\infty (\mathcal{M})  \right\}.$$

An operator $\mathcal{L}$ is said to be a generator of the stochastic transport equation, if for all $F\in \cylindricalFunctions$, and for any $L^2$ or white noise solutions $u$ to the stochastic transport equations,
\begin{equation}
M^F_t=F (u_t)-F (u_0)-\int_{0}^{t}\mathcal{L}F (u_s)\der s 
\end{equation}
is a  martingale.
\begin{prop}\label{generator}
	The generator $\mathcal{L}$ of the stochastic transport equation driven by noise with covariance operator $Q$ is given as follows:
	for any $n>0$, $f\in C^2_b (\mathbb{R}^n)$,$F (\cdot)=f ( (\cdot,\phi_1),\dots, (\cdot,\phi_j))\in\cylindricalFunctions$, the operator $\mathcal{L}$ acting on $F$ is given by 
	\begin{equation}
	 (\mathcal{L}F) (u)=\sum_{k=1}^{n}F_k (u) (u, \frac{1}{2}L\phi_k)+\frac{1}{2}\sum_{k,l=1}^{n}F_{kl} (u) (u\nabla\phi_k,Q (u\nabla\phi_l)),
	\end{equation}
	where $F_l (u)=\partial_l f ( (u,\phi_1),\dots, (u,\phi_n)), F_{kl}=\partial_k\partial_l f ( (u,\phi_1),\dots, (u,\phi_n))$.
\end{prop}
\begin{proof}
	Suppose $u_t$ is a solution to the stochastic transport equation.
	By It\^{o}'s formula,
	\begin{align*}
	\der F (u_t)=&\sum_{k=1}^{n}F_k (u_t) \left ( (u_t,\nabla\phi_k\cdot\der W_t)+\frac{1}{2} (u_t,L\phi_k)\der t \right)\\&+\frac{1}{2}\sum_{k,l=1}^{n}F_{kl} (u_t)\der\quavar{ (u,\phi_k), (u,\phi_l)}_t.
	\end{align*}
	Then by (\ref{quav-test}), 
	\[M^F_t=F (u_t)-F (u_0)-\int_{0}^{t}\mathcal{L}F (u_s)\der s\] is a  martingale.
\end{proof}

\begin{thm}[Convergence of generators]\label{thm-generator-cvgc}
	Suppose $\generator{N}$ is the generator of the stochastic transport equation driven by noise with covariance operator $Q^{ (N)}$, and the noises are properly tuned with (\ref{condition-diagonal}), (\ref{condition-hs}) satisfied. Then for any $F\in\cylindricalFunctions$, $\generator{N} F$ converges in $L^p (\mu)$ ($\mu$ is the law of white noise) for all $p\in[1,\infty)$ to $\generator{\infty}F$, with $\generator{\infty}$ defined by 
	\begin{equation}
	\generator{\infty}F (u)=\sum_{k=1}^{n}F_k (u) (u,\frac{1}{2}\Delta\phi_k)+\frac{1}{2}\sum_{k,l=1}^{n}F_{kl} (u) (\nabla\phi_k,\nabla\phi_l),
	\end{equation}
	where the notations $F_k,F_{kl}$ are the same as in Proposition \ref{generator}.
\end{thm}
\begin{proof}
	Let $\eta$ be a space white noise.
	As $F_k$, $F_{kl}$ are bounded, it suffices to show that
	\begin{equation}\label{-1cvgc}
	 (\eta,L^{ (N)}\phi_k)\longrightarrow (\eta, \Delta \phi_k)
	\end{equation} and
	\begin{equation}\label{-2cvgc}
	 (\eta\boxtimes\eta, Q^{ (N)}\cdot (\nabla\phi_k\boxtimes\nabla\phi_l))\longrightarrow  (\nabla\phi_k,\nabla\phi_l)
	\end{equation} in $L^p$ as $N$ goes to infinity.
	
	With Proposition \ref{prop-wn-moment}, we can get (\ref{-1cvgc})\: if $L^{ (N)}\phi_k$ converges in $L^2 (\manifold)$ to $\Delta \phi_k$, which is guaranteed by condition (\ref{condition-diagonal}).
	
	To get (\ref{-2cvgc}), it suffices to show for any $\phi\in C^\infty (\manifold)$,
	\begin{equation}\label{-3cvgc}
	 (\eta\boxtimes\eta,Q^{ (N)}\cdot  (\nabla\phi\boxtimes\nabla\phi))\longrightarrow (\nabla \phi,\nabla \phi),
	\end{equation}
	since
	\begin{align*}
	 (&\eta\boxtimes\eta,\nabla\phi_k\boxtimes\nabla\phi_l)\\= &\frac{\left ( \eta \boxtimes \eta, \big (\nabla (\phi_k+\phi_l)\boxtimes\nabla (\phi_k+\phi_l)-\nabla (\phi_k-\phi_l)\boxtimes\nabla (\phi_k-\phi_l)\big) \right)}{4}.
	\end{align*} 
	
	By Proposition \ref{prop-est-wn-quadratic}, 
	\[
	\Expectation[ (\eta\boxtimes\eta, Q^{ (N)}\cdot (\nabla\phi\boxtimes\nabla\phi))]=\int_{\manifold}\nabla\phi (x)\cdot A^{ (N)} (x)\nabla\phi (x)\der x. \]
	By condition (\ref{condition-diagonal}), the expectation converges to the desired limit $ (\nabla\phi,\nabla\phi)$.
	
	Now we need to control the $p$-th moment with the help of Proposition \ref{prop-est-wn-quadratic}. Let  $$E_N:=\Expectation[ (\eta\boxtimes\eta, Q^{ (N)}\cdot (\nabla\phi\boxtimes\nabla\phi))].$$
    Then
	\begin{equation}\label{pf-eq-01}
	\Expectation| (\eta\boxtimes\eta, Q^{ (N)}\cdot (\nabla\phi\boxtimes\nabla\phi))-E_N|^p \lesssim_p \int_{\manifold\times\manifold} \big|\nabla\phi (x)\cdot Q^{ (N)} (x,y)\nabla\phi (y)\big|^2\der x \der y
	\end{equation}
	Now we control the RHS:
	\begin{align*}
	&\int_{\manifold\times\manifold} \big|\nabla\phi (x)\cdot Q^{ (N)} (x,y)\nabla\phi (y)\big|^2\der x \der y
	\\ \leq&\int_{\manifold\times\manifold} \lVert\nabla\phi \rVert^2_{L^\infty} |Q^{ (N)} (x,y)|^2 \der x \der y
	 \\ = & \lVert\nabla\phi \rVert^2_{L^\infty} \|Q^{ (N)}\|_{L^2 (\manifold\times\manifold, \tanbdl\boxtimes\tanbdl)}.
	\end{align*}
	By condition (\ref{condition-hs}), the $L^2$ norm of $Q^{N}$ goes to zero. Therefore, by (\ref{pf-eq-01}), we know that $ (\eta\boxtimes\eta, Q^{ (N)}\cdot (\nabla\phi\boxtimes\nabla\phi))-E_N$ converges to zero in $L^p$. Together with  the convergence of $E_N$ to $ (\nabla\phi,\nabla\phi)$, we prove the convergence of (\ref{-3cvgc}) in $L^p$, which implies (\ref{-2cvgc}).
	
\end{proof}

\begin{rmk}
	Note that when $u$ is in $L^2 (\manifold)$, $$\generator{N}F (u)\rightarrow \sum_{k=1}^{n}F_k (u,\frac{1}{2}\Delta \phi_k) $$ as $N\longrightarrow\infty$, since $$ (u\nabla\phi_k,Q (u\nabla\phi_l))\leq \LtwoMnorm{u}\lVert \nabla \phi_k \rVert_{L^\infty}
	\lVert \nabla \phi_l \rVert_{L^\infty} \opnorm{Q^{ (N)}}$$ goes to zero as $N$ goes to infinity.
	
	It does not contradict Theorem \ref{thm-generator-cvgc} because $L^2 (\manifold)$ is a null set with respect to the law of space white noise.
	
	The difference of the generator limits in these two cases implies different behaviors of the limits of the solutions. 
\end{rmk}

\subsubsection{Convergence of solutions}

We keep the notations of previous subsubsection. Before proving the convergence of white noise solutions, let us take a look at the candidate limit determined by the limit generator.

\begin{lma}\label{lma: uniqueness of martingale solution}
    Let $\alpha<-d/2, T>0$. Suppose $\ulimit\in C([0,T],H^\alpha(\manifold))$ solves the following martingale problem for $\generator{\infty}$ i.e., for any $F\in\cylindricalFunctions$, the process $M^F_t$ defined by
\begin{equation}\label{mtg-limit}
M^F_t=F (\ulimit_t)-F (\ulimit_0)-\int_{0}^{t}\generator{\infty}F (\ulimit_s)\der s
\end{equation} is a martingale and the distribution of $\ulimit_0$ is space white noise. Then the distribution of $\ulimit$ is uniquely determined and given by the solution to the following stochastic heat equation:
\begin{equation}
\label{eq-wn-limit}
\dfrac{\partial}{\partial t} \ulimit = \frac{1}{2}\Delta\ulimit +  (-\Delta)^{1/2}\stwhitenoise,
\end{equation}
where $\stwhitenoise$ is a space-time white noise.
\end{lma}
\begin{proof}
    Let $f_n \in C_b^2(\mathbb{R})$, such that $f_n=1$ on $[-n,n]$ and $|f_n'|\leq1,|f_n''|\leq 1.$. Take $F^n(u)=f_n((u,\phi))\in\cylindricalFunctions$ in the martingale problem. Since $F^n,\generator{\infty}F^n$ are uniformly bounded in $L^2(\mu)$, as $n$ goes to infinity, we get that 
    \[\tilde{W}_t(\phi)=(\ulimit_t,\phi)-(\ulimit_0,\phi)-\int_0^t (\ulimit_s,\Delta \phi)\der s\] is a continuous martingale.
    Now we compute their quadratic variations. Let $\varphi$ also be a smooth function. 
    By It\^o's formula,
    \begin{align*}
        &(\ulimit_t,\phi)(\ulimit_t,\varphi)-(\ulimit_0,\phi)(\ulimit_0,\varphi)
        \\&=\int_0^t \frac{1}{2}\left((\ulimit_s,\phi)(\ulimit_s,\Delta\varphi)+(\ulimit_s,\phi)(\ulimit_s,\Delta\varphi)\right)\der s+\quavar{\tilde{W}(\phi),\tilde{W}(\varphi)}_t
        \\&+\int_0^t\frac{1}{2}\left( (\ulimit_s,\Delta\phi)\der \tilde{W}_s(\varphi) +(\ulimit_s,\Delta\varphi)\der \tilde{W}_s(\phi) \right).
    \end{align*}
    
    Similarly by taking $F^n(u)=f_n((u,\phi))f_n((u,\varphi))$,
    we know that 
    \begin{align*}
        &(\ulimit_t,\phi)(\ulimit_t,\varphi)-(\ulimit_0,\phi)(\ulimit_0,\varphi)
        \\&-\int_0^t \frac{1}{2}\left((\ulimit_s,\phi)(\ulimit_s,\Delta\varphi)+(\ulimit_s,\phi)(\ulimit_s,\Delta\varphi)+(\nabla \phi, \nabla \varphi)\right)\der s
    \end{align*}
    is a  martingale.
    Therefore, $\quavar{\tilde{W}(\phi),\tilde{W}(\varphi)}_t=t(\nabla \phi, \nabla \varphi)=t(\phi,\Delta \varphi)$. $\tilde{W}_t=\ulimit_t-\ulimit_0+\frac{1}{2}\int_0^t\Delta u_s\der s$ is an adapted generalized Wiener process with covariance operator $\Delta$. $\partial_t\tilde{W}$ has the same law with $(-\Delta)^{1/2}\xi$, where $\xi$ is a space time white noise. Therefore, $\ulimit$ can be given as a solution to (\ref{eq-wn-limit}). $\tilde{W}$ is in $H^{\alpha}(\manifold)$ by Proposition \ref{prop: finite trace} and $\Delta$ generates a contraction semigroup on $H^\alpha(\manifold)$, by \cite{da2014stochastic}, we have weak uniqueness and $\ulimit_t=\int_0^te^{(t-s)\Delta}\der \tilde{W}_s+\ulimit_0$.  Therefore, the law of $\ulimit$ is uniquely determined.
    
\end{proof}

\begin{thm}
	Suppose $T>0,$ $ \alpha \in [0,1/2)$, $\beta<-\frac{d}{2}-2\alpha$ and the noises satisfy conditions (\ref{condition-diagonal}) and (\ref{condition-l2}). Then the  white noise solutions $u^{ (N)}$ driven by noises with covariant matrices $Q^{ (N)}$ on $[0,T]$ converge in distribution to the white noise solution to (\ref{eq-wn-limit}) as processes in $C^{\alpha} ([0,T],H^{\beta})$.
\end{thm}
\begin{proof}
	\newcommand{\calphahbeta}{C^{\alpha} ([0,T],H^{\beta})}
	We use a tightness argument to get the convergence result. We first show the tightness of the processes, and then for any subsequence we can take a convergent sub-subsequence. Then we show that the limit satisfies the martingale problem (\ref{mtg-limit}). Therefore, it must be stationary solution to (\ref{eq-wn-limit}). Therefore, the series converge to the universal limit.
	
	Note that for $0\leq\alpha_1<\alpha_2<1/2$ and  $\beta_1<\beta_2$. $C^{\alpha_1} ([0,T],H^{\beta_1})$ is compactly embedded in $C^{\alpha_2} ([0,T],H^{\beta_2})$. Therefore, the tightness can be derived from moment bound of H\"older norm  (Corollary \ref{cor-holder-moment} of the Theorem \ref{thm-regularity-moment}). The constant can be chosen uniformly in $N$, as condition  \eqref{condition-diagonal} guarantees that $A^{(N)}$ and $(A^{(N)})^{-1}$ are uniformly bounded in the $C^k$ norm for any integer $k$. 
	
	Then by Prokhorov's theorem, we know that for any subsequence of $u^{ (N)}$, there is a convergent subsubsequence. $\tilde u^{ (N_n)}$ satisfies the martingale problem  with generator $\generator{N_n}$ with respect to its natural filtration.

	To show that the limit satisfies the martingale problem (\ref{mtg-limit}) with respect to the natural filtration, it suffices to show that for any $F\in \cylindricalFunctions, s\in (0,T)$, $\varphi\in C_b\left (C^\alpha ([0,s],H^\beta)\right)$,$t\in[s,T]$,
	\begin{equation}
	\newexp\left[ \left (F (\tilde{u}^{ (\infty)}_t)-F (\tilde{u}^{ (\infty)}_s)-\int_{s}^{t} \generator{\infty} F (\tilde{u}^{ (\infty)}_r) \der r \right) \varphi  (\tilde{u}^{ (\infty)}) \right]=0.
	\end{equation}
	
	Since $\tilde{u}^{ (N_n)}_t$ satisfies the martingale problem with generator $\generator{N_n}$ with respect to the filtration $\tilde{F}^n_t$,
	\begin{equation}
	\newexp\left[ \left (F (\tilde{u}^{ (N_n)}_t)-F (\tilde{u}^{ (N_n)}_s)-\int_{s}^{t} \generator{N_n} F (\tilde{u}^{ (N_n)}_r) \der r \right) \varphi  (\tilde{u}^{ (N_n)}) \right]=0.
	\end{equation}
	
	Therefore,
	\begin{align*}
	&\left|\newexp\left[\left (F (\tilde{u}^{ (N_n)}_t)-F (\tilde{u}^{ (N_n)}_s)-\int_{s}^{t} \generator{\infty} F (\tilde{u}^{ (N_n)}_r) \der r \right) \varphi  (\tilde{u}^{ (N_n)}) \right]\right|
	\\ =
	& \left|\newexp \left (\int_{s}^{t}  (\generator{N_n}-\generator{\infty}) F (\tilde{u}^{ (N_n)}_r) \der r \right) \varphi  (\tilde{u}^{ (N_n)}) \right|
	\\
	\leq& \lVert \varphi \rVert_{L^\infty} \int_{s}^{t} \newexp| (\generator{N_n}-\generator{\infty}) F (\tilde{u}^{ (N_n)}_r) |\der r
	\\ =& \lVert \varphi \rVert_{L^\infty}  (t-s) \lVert ( \generator{N_n}-\generator{\infty}) F \rVert_{L^1 (\mu)}.
	\end{align*}
	
	By Theorem \ref{thm-generator-cvgc}, $\lVert \generator{N_n}-\generator{\infty}) F \rVert_{L^1 (\mu)}$ goes to zero, so
	$$ \newexp\left[\left (F (\tilde{u}^{ (N_n)}_t)-F (\tilde{u}^{ (N_n)}_s)-\int_{s}^{t} \generator{\infty} F (\tilde{u}^{ (N_n)}_r) \der r \right) \varphi  (\tilde{u}^{ (N_n)}) \right] \longrightarrow 0, $$ as $n$ goes to infinity.
	
	It remains to show that $$\left (F (\tilde{u}^{ (N_n)}_t)-F (\tilde{u}^{ (N_n)}_s)-\int_{s}^{t} \generator{\infty} F (\tilde{u}^{ (N_n)}_r) \der r \right) \varphi  (\tilde{u}^{ (N_n)})$$ converges in $L^1 (\newprob)$ to $$\left (F (\tilde{u}^{ (\infty)}_t)-F (\tilde{u}^{ (\infty)}_s)-\int_{s}^{t} \generator{\infty} F (\tilde{u}^{ (\infty)}_r) \der r \right) \varphi  (\tilde{u}^{ (\infty)}).$$
	
	The almost sure convergence follows from the almost sure convergence of $\tilde{u}^{ (N_n)}$ to $\tilde{u}^{ (\infty)}$. Therefore, we only need to show the uniform integrability, which can be obtained by uniform bound in $L^p (\newprob)$ for some $p>1$. As $F$ and $\varphi$ are bounded, all we need to get is the uniform $L^p (\newprob)$ boundedness of $\int_s^t\generator{\infty} F (\tilde{u}^{ (N_n)}_r) \der r$. By Minkowski's inequality, 
	\begin{align*}
	\left\lVert\int_s^t\generator{\infty} F (\tilde{u}^{ (N_n)}_r) \der r\right\rVert_{L^p (\newprob)}
	\leq& \int_s^t\left\lVert \generator{\infty} F (\tilde{u}^{ (N_n)}_r)\right\rVert_{L^p (\newprob)} \der r
	\\=& (t-s)\lVert \generator{\infty} F \rVert_{L^p (\mu)},
	\end{align*}
	since for each $r\in[0,T]$, $\tilde{u}^{ (N_n)}_r$ has the distribution of white noise. As $\generator{\infty} (\eta)$  (where $\eta$ is a space white noise) is a finite sum of some bounded functions times polynomials of Gaussian random variables, $\lVert \generator{\infty} F \rVert_{L^p (\mu)}$ is finite due to Gaussian hypercontractivity. Hence, we get the $L^1$ convergence of $$\left (F (\tilde{u}^{ (N_n)}_t)-F (\tilde{u}^{ (N_n)}_s)-\int_{s}^{t} \generator{\infty} F (\tilde{u}^{ (N_n)}_r) \der r \right) \varphi  (\tilde{u}^{ (N_n)})$$ to $$\left (F (\tilde{u}^{ (\infty)}_t)-F (\tilde{u}^{ (\infty)}_s)-\int_{s}^{t} \generator{\infty} F (\tilde{u}^{ (\infty)}_r) \der r \right) \varphi  (\tilde{u}^{ (\infty)}).$$  Therefore, $\tilde u^{ (\infty)}$ solves the martingale problem with the generator $\generator{\infty}$, and its initial distribution is a space white noise. By Lemma \ref{lma: uniqueness of martingale solution}, the distribution of the limit is uniquely determined by the martingale problem, together with tightness, we get the convergence in distribution for the whole sequence in $\calphahbeta$.  
	
\end{proof}

\subsection{Convergence to deterministic heat equation with regular initial data}\label{s-cvg-regular}

We can analogously use the tightness argument and the convergence of generator to show that with fixed $L^2$ initial data, the solutions to the stochastic transport equations converge to the deterministic heat equation. However, this method does not give any information about the convergence rate. In this subsection, we will give some quantitative estimates on the speed of convergence, which are proved for stochastic transport equations on a torus in \cite{flandoli2024quantitative}. 

Now we consider a solution to the stochastic heat equation driven by noise of covariance $Q$ with deterministic  initial value $u_0\in L^2 (\manifold)$. We give some quantitative bounds to the process $v_t$ of the difference between $u_t$ and its expectation (the solution to the heat equation) $v_t=u_t-P_t u_0$, where $P_t$ is the semigroup generated by $\frac{1}{2}L=\frac{1}{2}\nabla \cdot (A \nabla)$. With these quantitative bounds, we can show $v^{ (N)}$, which is similarly defined as $v^{ (N)}_t= u_t^{ (N)}-P^{ (N)}_t u_0$, converges to zero as $N$ goes to infinity. 

By the mild formulation, 
\begin{equation}
v_t=\int_{0}^{t}P_{t-s} (\nabla u_s \cdot \der W_s).
\end{equation}

First we show that $v_t$ is small in the weak sense:
\begin{thm}\label{thm-qt-weak-bd}
	For every $\phi\in L^2 (\manifold)$, and $t\geq 0$,\begin{equation}\label{eq-l2-weakbd}
	\Expectation[ (v_t,\phi)^2] \leq \frac {C (A)}{2} \opnorm{Q} \lpmani{u_0}{\infty}^2  \LtwoMnorm{\phi}^2,
	\end{equation}
	with $$C (A) =\dfrac{1}{\inf\limits_{\substack{X\in L^2 (\manifold,\tanbdl),\\ \lVert X \rVert_{L^2}=1 }} (X,AX)}.$$
	
	Moreover, for any $p\in (1,\infty)$, 
	\begin{equation}\label{eq-lp-weakbd}
	\Expectation[ (v_t,\phi)^p]^{1/p} \lesssim_p  (C (A)/2)^{1/2} \opnorm{Q}^{1/2} \lpmani{u_0}{\infty}  \LtwoMnorm{\phi}.
	\end{equation}
\end{thm}

\begin{proof}
	We only need to prove for smooth $\phi,u_0$ and the general case follows from the density argument.
	
	For fixed $t$, let us consider the following martingale $w_s=\int_{0}^{s}P_{t-r} (\nabla u_r \cdot \der W_r)$. Then $w_t=v_t$, $w_0=0$.
	\begin{align*}
	\der  (w_r,\phi)=& (P_{t-r} \nabla u_r \cdot \der W_r, \phi)
	\\ =& (\nabla u_r \cdot \der W_r, P_{t-r} \phi)
	\\ =&- (u_r \cdot \der W_r, \nabla  P_{t-r} \phi)
	\\= &-\int_{\manifold} u_r (x) W^i_r (x) \partial_i  (P_{t-r} \phi) (x)\der x.
	\end{align*}
	\begin{align*}
	&\dfrac{\der}{\der r} \quavar{ (w,\phi)}_r \\= &\int_{\manifold\times\manifold} Q^{ij} (x,y)\partial_i  (P_{t-r} \phi) (x)\partial_j  (P_{t-r} \phi) (y)u_r (x)u_r (y)\der x \der y
	\\= & (u_r \nabla P_{t-r}\phi, Q  (u_r\nabla P_{t-r}\phi))
	\\ \leq & \opnorm{Q} \lpmani{u_r}{\infty}^2 \lVert \nabla P_{t-r}\phi \rVert_{L^2}^2
	\\ \leq & C(A) \opnorm{Q} \lpmani{u_0}{\infty}^2  (\nabla P_{t-r}\phi, A\nabla P_{t-r}\phi)
	\\ = & - C(A) \opnorm{Q}\lpmani{u_0}{\infty}^2 (P_{t-r}\phi,L P_{t-r}\phi)
	\\ = &  \frac{C(A)}{2} \opnorm{Q}\lpmani{u_0}{\infty}^2\frac{\der}{\der r} \LtwoMnorm{P_{t-r}\phi}^2.
	\end{align*}
	
	Therefore,
	\begin{align*}
	    \quavar{ (w,\phi)}_t\leq & \frac{C(A)}{2} \opnorm{Q}\lpmani{u_0}{\infty}^2 (\LtwoMnorm{\phi}^2-\LtwoMnorm{P_t\phi}^2)
        \\ 
        \leq&\frac{C(A)}{2}  \opnorm{Q}\lpmani{u_0}{\infty}^2 \LtwoMnorm{\phi}^2.
	\end{align*}

	Since $v_t=w_t$, and the fact that the second moment of a martingale  (starting from 0) at time $t$ equals to the expectation of its quadratic variation at time $t$, we get (\ref{eq-l2-weakbd}). With Burkholder-Davis-Gundy inequality we get (\ref{eq-lp-weakbd}).
\end{proof}

In addition to the time-wise weak convergence, we can get a process level convergence rate estimate with the techniques in the proof of regularity results for white noise solutions:

\begin{thm}\label{thm-quanti-bd-1}
	Suppose $T>0,\kappa\in (0,1)$, $p\in[2,\infty)$. 
	Then there exists a constant $C=C(p,T,\manifold, A,\kappa)>0$ , such that for any $t,s\in[0,T]$, 
	\begin{equation}\label{bd-quant-1}
	\lpnormprob{\hnormm{v_t-v_s}{-\frac{d}{2}-\kappa}}{p} \leq C \opnorm{Q}^{1/2}|t-s|^{\kappa/2}\LtwoMnorm{u_0}. 
	\end{equation}
\end{thm}
\begin{proof}

Note that
\begin{align*}
    v_t-v_s= &\int_{0}^{t} P_{t-r} \nabla u_r \cdot \der W_r-\int_{0}^{s} P_{s-r} \nabla u_r 
    \cdot \der W_r
    \\=&\int_0^s (P_{t-s}-1)P_{s-r}\nabla u_r\cdot \der W_r + \int_{s}^t P_{t-r} \nabla u_r \cdot \der W_r.
\end{align*}

Using Proposition \ref{prop moment estimate via quadratic}, we control the first term by
\begin{align*}
    &\lpnormprob{\anorm{\int_0^s (P_{t-s}-1)P_{s-r}\nabla u_r\cdot \der W_r}{-d/2-\kappa}}{p}^2
    \\&\lesssim_p  \int_0^s\lpnormprob{\tr_{H^{-d/2-\kappa}_{A}}((P_{t-r}-P_{s-r})\nabla\multiplyfunc{u_r})Q((P_{t-r}-P_{s-r})\nabla\multiplyfunc{u_r})^\ast}{p/2} \der r
    \\&\leq \int_0^s\lpnormprob{\tr(Q\circ\multiplyfunc{u_r}\nabla(1-L)^{-d/2-\kappa}(P_{s-r}-P_{t-r})^2 \nabla^\ast \multiplyfunc{u_r})}{p/2} \der r
\end{align*}
 By Proposition \ref{prop trace(AB) < opnorm A trace B},
 \begin{align*}
     &\tr(Q \circ \multiplyfunc{u_r}\nabla(1-L)^{-d/2-\kappa}(P_{s-r}-P_{t-r})^2 \nabla^\ast \multiplyfunc{u_r}) 
     \\ \leq &\opnorm{Q}\tr(\multiplyfunc{u_r}\nabla(1-L)^{-d/2-\kappa}(P_{s-r}-P_{t-r})^2 \nabla^\ast \multiplyfunc{u_r})
     \\ \leq &\opnorm{Q}\tr(\multiplyfunc{u_r}\nabla (1-L)^{-d/2-\kappa}(P_{2(s-r)}-P_{2(t-r)})\nabla^\ast\multiplyfunc{u_r}).
 \end{align*}
 By Proposition \ref{prop trace uqu} and Lemma \ref{lma:heat kernel estimate C2},
 \begin{align*}
     &\tr(\multiplyfunc{u_r}\nabla (1-L)^{-d/2-\kappa}(P_{2(s-r)}-P_{2(t-r)})\nabla^\ast\multiplyfunc{u_r})
     \\& = \int_\manifold u_r^2(x) \nabla_x\nabla_y (\kernelqAalpha{-d/2-\kappa}{s-r}-\kernelqAalpha{-d/2-\kappa}{t-r})(x,y)|_{y=x} \der x
     \\ &\leq \LtwoMnorm{u_r}^2 \| \kernelqAalpha{-d/2-\kappa}{s-r} -\kernelqAalpha{-d/2-\kappa}{t-r}\|_{C^2(\manifold\times\manifold)}
     \\ &\lesssim_{T,A,\manifold}  \LtwoMnorm{u_0}^2 ((s-r)^{\kappa-1}-(t-r)^{\kappa-1}),
 \end{align*}
    where $q_t^{A,\alpha}$ is the integral kernel of the operator $(1-L)^\alpha e^tL$.
 Therefore
 \begin{align*}
    &\lpnormprob{\anorm{\int_0^s (P_{t-s}-1)P_{s-r}\nabla u_r\cdot \der W_r}{-d/2-\kappa}^2}{p}
    \\ &\lesssim_{p,A,T,\manifold}  \opnorm{Q} \LtwoMnorm{u_0}^2 \int_0^s((s-r)^{\kappa-1}-(t-r)^{\kappa-1})\der r 
    \\ &\lesssim_\kappa   \opnorm{Q} \LtwoMnorm{u_0}^2 (t-s)^\kappa.
 \end{align*}
Then we bound the second term by Proposition \ref{prop moment estimate via quadratic}, Proposition \ref{prop trace(AB) < opnorm A trace B}, Proposition \ref{prop trace uqu} and Lemma \ref{lma:heat kernel estimate C2}:
 \begin{align*}
        &\lpnormprob{\anorm{\int_s^t P_{t-r} \nabla u_r \cdot \der W_r}{-d/2-\kappa}}{p}^{2}
        \\ &\lesssim_p \int_s^t \lpnormprob{\tr_{H_A^{-d/2-\kappa}}(P_{t-r}\nabla ^\ast \multiplyfunc{u_r} Q  \multiplyfunc{u_r}  \nabla P_{t-r})}{p/2} \der r
        \\ &=\int_s^t\lpnormprob{\tr(Q\circ\multiplyfunc{u_r}\nabla(1-L)^{-d/2-\kappa}P_{2(t-r)} \nabla^\ast \multiplyfunc{u_r})}{p/2} \der r
        \\ &\leq \int_s^t\opnorm{Q} \LtwoMnorm{u_r}^2 \| \kernelqAalpha{-d/2-\kappa}{t-r}\|_{C^2(\manifold\times\manifold)} \der r
        \\ &\leq  \opnorm{Q}\LtwoMnorm{u_0}^2\int_s^t (t-r)^{\kappa-1} \der r
        \\ \lesssim_\kappa & \opnorm{Q}\LtwoMnorm{u_0}^2 (t-s)^\kappa.
    \end{align*}

    As $H^{-d/2-\kappa}_A$  and $H^{-d/2-\kappa}$ are equivalent and the constants  can be chosen uniformly in $A$ if $A$ and $A^{-1}$ is uniformly bounded in the $C^k$ norm for some sufficiently large $k$, we get the bound (\ref{bd-quant-1}).
    
\end{proof}

\begin{rmk}
	Note that $A^{ (N)}$ converges to the identity matrix in $C^\infty$ (condition (\ref{condition-diagonal})), we can have a uniform $C$ for all solutions $u^{ (N)}$ in the estimate (\ref{bd-quant-1}). Therefore, Theorem \ref{thm-quanti-bd-1} gives a quantitative bound to the convergence rate.
\end{rmk}

\begin{thm}
		Suppose  $p\in[2,\infty)$.
		Then there exists a constant $C=C(\manifold,A)>0$ such that for any $t,s\in[0,\infty)$,  
		\begin{equation*}
		\lpnormprob{\hnormm{v_t-v_s}{-1}}{p} \leq C |t-s|^{1/2} \LtwoMnorm{u_0}.
		\end{equation*}
        Moreover, the constant can be chosen uniformly in $A$ provided that the $L^\infty$ norm of $A$ and $A^{-1}$ are uniformly bounded.
\end{thm}
\begin{proof}
	First note that 
	$$v_t-v_s= (P_{t-s}-\identity)v_s +\int_s^t P_{t-r} \nabla u_r \cdot \der W_r.$$
	The first term can be easily bounded by heat kernel estimates and the $L^2$ boundedness of $v_s$:
    \begin{equation}\label{v -1 bd 1}
        \hnormm{ (P_{t-s}-\identity)v_s}{-1}\lesssim_{A,\manifold} (t-s)^{1/2}\LtwoMnorm{v_s}\lesssim \LtwoMnorm{u_0} (t-s)^{1/2}.
    \end{equation}
	
	The stochastic convolution term can be bounded following the same arguments in Theorem \ref{thm l2 h-1 bd},
    \begin{equation} \label{v -1 bd 2}
        \lpnormprob{\hnormm{\int_s^t P_{t-r} \nabla u_r \cdot \der W_r}{-1}}{p} \lesssim_{A,\manifold,p} \LtwoMnorm{u_0}(t-s)^{1/2}.
    \end{equation}
    Combining (\ref{v -1 bd 1}) and (\ref{v -1 bd 2}), we get the desired bound.
	
\end{proof}

Noting that $\LtwoMnorm{v_t-v_s}\leq \LtwoMnorm{v_t}+\LtwoMnorm{v_s}\leq 4\|u_0\|_{L^2(\manifold)}$, with interpolation of the Sobolev spaces, we can easily get the following corollary:

\begin{cor}\label{quantitative-stoch}
		Suppose $T>0,\kappa\in (0,1)$, $p\in[2,\infty), \alpha\in[0,1]$. 
		Then there exists a constant $C>0$ which may depend on $\kappa,p$, the manifold $\manifold$ and the $C^k$ norms of $A$ and $A^{-1}$ for some large enough $k$, such that for any $t,s\in[0,T]$, 
		\begin{equation}\label{bd-quali-1}
		\lpnormprob{\hnormm{v_t-v_s}{-\frac{\alpha d}{2}-\kappa}}{p} \leq C \opnorm{Q}^{\alpha/2}|t-s|^{\kappa/2}\LtwoMnorm{u_0}.
		\end{equation}
		
		With Kolmogorov continuity criterion, for $\varepsilon\in (0,\kappa/2]$, we can get the convergence rate of $v^{ (N)}=u^{ (N)}-P_t^{ (N)}u^{ (N)}$:
		\begin{equation}
		\lpnormprob{\lVert v^{ (N)} \rVert_{C^{\frac{\kappa}{2}-\varepsilon} ([0,T],H^{-\frac{\alpha d}{2}-\kappa})}}{p} \lesssim \opnorm{Q^{ (N)}}^{\alpha/2}\LtwoMnorm{u_0}.
		\end{equation}
\end{cor}

Now we provide the convergence rate of the deterministic part $\bar{u}_t=P_t u_0$.
We have the following estimate:
\begin{prop}\label{quantitative-deter}
	Let $\kappa\in [0,1), T>0$. Then the deterministic error $w_t:=P_t u_0 - e^{t\Delta} u_0$ can be bounded by
	\begin{equation}\label{whoelder}
	 \hnormm{w_t-w_s}{-\kappa} \lesssim_{A,T} \lVert A-I\rVert_{L^\infty}|t-s|^{\kappa/2}\LtwoMnorm{u_0},
	\end{equation}
	where the implicit constant can be chosen uniformly in $A$, provided that $A$ and $A^{-1}$ are uniformly bounded in $L^\infty$ norm.
\end{prop}
\begin{proof}
	
	Note that 
	\[ (\partial_t-\Delta)w_t= (L-\Delta)P_t u_0. \]
	Therefore,
	\begin{equation}\label{wmild}
	w_t-w_s= (e^{ (t-s)\Delta}-\identity)w_s+\int_{s}^{t} e^{ (t-r)\Delta} (L-\Delta)P_r u_0 \der r.
	\end{equation}

	First, we take $s=0$ and bound $w_t$. We only need to bound the integral on the right hand side since $w_0=0$.

	\begin{align*}
	&\hnormm{ e^{ (t-r)\Delta} (L-\Delta)P_r u_0}{-\kappa}
	\\ \leq & \lVert e^{ (t-r)\Delta} \rVert_{H^{-1}\rightarrow H^{-\kappa}} \lVert L-\Delta \rVert_{H^1\rightarrow H^{-1}} \lVert P_r \rVert_{L^2\rightarrow H^{1}} \lVert u_0 \rVert_{L^2}.
	\end{align*}
	
	Note that $\lVert e^{ (t-r)\Delta} \rVert_{H^{-1}\rightarrow H^{-\kappa}}\leq  (t-r)^{\frac{\kappa-1}{2}}$,
	\begin{equation}\label{h1-h-1}
	\lVert L-\Delta \rVert_{H^1\rightarrow H^{-1}} \leq \lVert A-I \rVert_{L^\infty},
	\end{equation}
	\begin{equation}\label{prl2-h1}
	\lVert P_r \rVert_{L^2\rightarrow H^{1}} \lesssim_{A} r^{-1/2},
	\end{equation}
	where (\ref{h1-h-1}) follows from 
	\[ (\phi_1,  (L-\Delta)\phi_2)\leq \lVert A-I \rVert_{L^\infty} \ltwotan{\nabla\phi_1}\ltwotan{\nabla\phi_2} \]
	 for any $\phi_1,\phi_2 \in C^\infty (\manifold)$ and (\ref{prl2-h1}) follows from spectral calculus,
	\[ \lVert  (1-L)^{1/2} P_r \rVert_{L^2\rightarrow L^2} \lesssim_T r^{-1/2} \] 	
	and the equivalence of the $H^1$ norm and $\anorm{\cdot}{1}$.
	
    We have
	\begin{align*}
	\hnormm{w_t}{-\kappa}& \lesssim_{{T,A}} \lVert A-I\rVert_{L^\infty}\LtwoMnorm{u_0}\int_{0}^{t}  (t-r)^{\frac{\kappa-1}{2}}s^{-1/2}\der r
	\\&  = \lVert A-I\rVert_{L^\infty}t^{\kappa/2}\LtwoMnorm{u_0}\int_{0}^{1} (1-s)^{\frac{\kappa-1}{2}}s^{-1/2}\der s
	\\& \lesssim_{\kappa}  \lVert A-I\rVert_{L^\infty}t^{\kappa/2}\LtwoMnorm{u_0}.
	\end{align*}
    
	In particular, taking $\kappa=0$, we have
	\begin{equation}
	\LtwoMnorm{w_t} \lesssim_{A,T,\kappa} \lVert A-I\rVert_{L^\infty}\LtwoMnorm{u_0}.
	\end{equation}
	
	Now we use (\ref{wmild}) to get (\ref{whoelder}).
	To bound the first term on the right hand side:
	\begin{align*}
	&\hnormm{ (e^{ (t-s)\Delta}-\identity)w_s}{-\kappa}
	\\ \lesssim & (t-s)^{\kappa/2}\LtwoMnorm{w_s}
	\\ \lesssim & (t-s)^{\kappa/2}\lVert A-I \rVert_{L^\infty}\LtwoMnorm{u_0}.
	\end{align*}
	Then it remains to bound the integral. Note that $r^{-1/2}\leq  (r-s)^{-1/2}$,
	\begin{align*}
	&\hnormm{\int_{s}^{t} e^{ (t-r)\Delta} (L-\Delta)P_r u_0 \der r}{-\kappa}
	\\ &\lesssim_{T,A}\lVert A-I\rVert_{L^\infty}\LtwoMnorm{u_0}\int_{0}^{t}  (t-r)^{\frac{\kappa-1}{2}} (r-s)^{-1/2}\der r
	\\ &=  \lVert A-I\rVert_{L^\infty} (t-s)^{\kappa/2}\LtwoMnorm{u_0}\int_{0}^{1} (1-s)^{\frac{\kappa-1}{2}}s^{-1/2}\der s
	\\ &\lesssim_{\kappa}  \lVert A-I\rVert_{L^\infty} (t-s)^{\kappa/2}\LtwoMnorm{u_0}.
	\end{align*}
\end{proof}
Together with Corollary \ref{quantitative-stoch}, we get the following corollary which gives a quantitative estimate to the convergence rate with the noises scaled as in Theorem \ref{thm scaling via mollified space time white noise}.\begin{cor}
	Suppose $\kappa\in (0,1), T>0, p\in[2,\infty), \varepsilon\in  (0, \kappa/2].$
    If the noises are given by renormalized mollified space-time white noise as in  Theorem  \ref{thm scaling via mollified space time white noise}, $u=u^h$ is the solution to the stochastic transport equation with $L^2$ initial value $u_0$ and $\tilde{u},\bar{u}^h$ are the solutions to the parabolic equations
    \begin{equation}
        \partial_t\tilde{u}=\frac{1}{2}\Delta \tilde{u},
    \end{equation}
    
    \begin{equation}
     \partial_t \bar {u}^h=\frac{1}{2}\nabla(A^h\nabla\bar u^h)   
    \end{equation}with the same initial value, which is the scaling limit of the stochastic transport equations.
    Then $v_t=u^h_t-\bar{u}^h_t$ is bounded by
    \begin{equation}\label{eq-bd-h-stoch}
        \lpnormprob{\lVert v \rVert_{C^{\frac{\kappa}{2}-\varepsilon} ([0,T],H^{-\frac{\alpha d}{2}-\kappa})}}{p} \lesssim h^{\alpha d/2}\LtwoMnorm{u_0},
    \end{equation} for $\alpha\in[0,1]$.
    
    When $d\geq 4$, 
    \begin{equation}\label{eq4tobeproved}
        \lpnormprob{\lVert{u^h-\tilde{u}}\rVert_{C^{\frac{\kappa}{2}-\varepsilon} ([0,T],H^{-2-\kappa})}}{p} \lesssim h^2\LtwoMnorm{u_0}.
    \end{equation}

    When $d=2,3$,
    \begin{equation}\label{eq4tobeproved d=2,3}
        \lpnormprob{\lVert{u^h-\tilde{u}}\rVert_{C^{\frac{\kappa}{2}-\varepsilon} ([0,T],H^{-d/2-\kappa})}}{p} \lesssim h^{d/2}\LtwoMnorm{u_0}.
    \end{equation}
\end{cor}
\begin{proof}
	 Note that $\lVert A^h-\identity \rVert_{L^\infty} \lesssim h^2$, $\opnorm{Q^h}\lesssim h^{d}$ by Theorem \ref{thm scaling via mollified space time white noise}. By Proposition \ref{quantitative-deter}, 
	\begin{equation}\label{eq4.1}
	\lVert \bar{u}^h-\tilde{u}\rVert_{C^{\kappa/2} ([0,T],H^{-\kappa})}\lesssim h^2.
	\end{equation} 
	Together with Corollary \ref{quantitative-stoch}, we get (\ref{eq-bd-h-stoch}).  When $d\geq 4$, taking $\alpha=4/d$,
	we get 
	\begin{equation}\label{eq4.2}
	\lpnormprob{\lVert{u^h-\bar{u}^h}\rVert_{C^{\frac{\kappa}{2}-\varepsilon} ([0,T],H^{-2-\kappa})}}{p} \lesssim h^2.
	\end{equation}
	Combining (\ref{eq4.1}) and (\ref{eq4.2}), we get (\ref{eq4tobeproved}).

    When $d=2,3$, taking $\alpha=1$ in Corollary \ref{quantitative-stoch}, we get
    \begin{equation}\label{eq4.2'}
	\lpnormprob{\lVert{u^h-\bar{u}^h}\rVert_{C^{\frac{\kappa}{2}-\varepsilon} ([0,T],H^{-d/2-\kappa})}}{p} \lesssim h^{d/2}.
	\end{equation}

    Combining (\ref{eq4.1}) and (\ref{eq4.2'}), we get (\ref{eq4tobeproved d=2,3}).
\end{proof}

	\appendix
    \section*{Acknowledgments}
       I would like to thank Nicolas Perkowski for his invaluable guidance and advice throughout the development of this work, and in particular for pointing out the connection to large-scale analysis in SPDEs. I also thank Zhengxiong Cao for helpful discussions, especially regarding the differential geometry aspects in the preliminaries. I am grateful to the anonymous referee for their careful reading and helpful suggestions, which have significantly improved the presentation of this paper.

This research was supported by the Deutsche Forschungsgemeinschaft (DFG) under Project-ID 410208580 – IRTG2544 (“Stochastic Analysis in Interaction”) and by Germany's Excellence Strategy – the Berlin Mathematics Research Center MATH+ (EXC-2046/1, Project ID: 390685689).

\bibliography{sn-bibliography}

\end{document}